\pdfoutput=1

\documentclass[12pt,reqno,a4wide]{amsart}

\usepackage{tikz}
\usetikzlibrary{arrows,backgrounds,calc}
\usepackage{calrsfs}
\usepackage[charter]{mathdesign}
\usepackage{amsthm, amsmath, amsfonts}
\usepackage{graphicx}
\usepackage{hyperref}
\usepackage{lineno}
\usepackage{enumerate}
\usepackage{mathtools}
\usepackage{subcaption}

\makeatletter
\let\pgfmathModX=\pgfmathMod@
\usepackage{pgfplots}%
\let\pgfmathMod@=\pgfmathModX
\makeatother
\pgfplotsset{compat=1.5}

\allowdisplaybreaks

\newtheorem{theorem}{Theorem}
\newtheorem{proposition}{Proposition}[section]
\newtheorem{corollary}[proposition]{Corollary}

\theoremstyle{remark}
\newtheorem*{remark}{Remark}

\theoremstyle{definition}
\newtheorem*{definition}{Definition}

\DeclareMathOperator{\Reg}{Reg}

\renewcommand{\MR}[1]{}

\newcommand{\FGF}{H}
\newcommand{\TODO}[1]%
{\par\fbox{\begin{minipage}{0.9\linewidth}\textbf{TODO:} #1\end{minipage}}\par}
\newcommand{\DLMF}[2]{\cite[\href{http://dlmf.nist.gov/#1.E#2}{#1.#2}]{NIST:DLMF:v1.0.13}}

\let\Re\relax

\let\Im\relax

\newcommand{\C}[0]{\mathbb{C}}
\newcommand{\R}[0]{\mathbb{R}}
\newcommand{\Z}[0]{\mathbb{Z}}
\renewcommand{\P}[0]{\mathbb{P}}
\newcommand{\E}[0]{\mathbb{E}}
\newcommand{\V}[0]{\mathbb{V}}
\newcommand{\N}[0]{\mathbb{N}}
\DeclareMathOperator{\Re}{Re}
\DeclareMathOperator{\Im}{Im}
\DeclareMathOperator{\Res}{Res}
\DeclareMathOperator{\rdeg}{rdeg}
\DeclarePairedDelimiter{\abs}{\lvert}{\rvert}

\newcommand{\upa}[0]{\mathnormal\uparrow}
\newcommand{\righta}[0]{\mathnormal\rightarrow}
\newcommand{\downa}[0]{\mathnormal\downarrow}
\newcommand{\lefta}[0]{\mathnormal\leftarrow}

\title[Reductions of Binary Trees and Lattice Paths]
{Reductions of Binary Trees and Lattice Paths\\ induced by the Register Function}

\author[B.~Hackl]{Benjamin Hackl}
\author[C.~Heuberger]{Clemens Heuberger}

\address[Benjamin Hackl, Clemens Heuberger]{Institut f\"ur Mathematik,
  Alpen-Adria-Uni\-ver\-si\-t\"at Klagenfurt, Universit\"atsstra\ss e
  65--67, 9020 Klagenfurt, Austria}
\email{\href{mailto:benjamin.hackl@aau.at}{benjamin.hackl@aau.at}}
\email{\href{mailto:clemens.heuberger@aau.at}{clemens.heuberger@aau.at}}
\thanks{B.~Hackl and C.~Heuberger are supported by the Austrian
  Science Fund (FWF): P~24644-N26 and by the Karl Popper Kolleg
  ``Modeling-Simulation-Optimization'' funded by the Alpen-Adria-Universit\"at Klagenfurt
  and by the Carinthian Economic Promotion Fund (KWF)}

\author[H.~Prodinger]{Helmut Prodinger}
\thanks{H.~Prodinger is supported by an incentive grant of the
  National Research Foundation of South Africa.}

\address[Helmut Prodinger]{Department of Mathematical
  Sciences, Stellenbosch University, 7602 Stellenbosch,
 South Africa}
\email{\href{mailto:hproding@sun.ac.za}{hproding@sun.ac.za}}

\thanks{This is the full version of the extended abstract~\cite{Hackl-Heuberger-Prodinger:2016:register}.}

\begin{document}

\maketitle
\begin{abstract}
  The register function (or Horton-Strahler number) of a binary tree is a well-known combinatorial parameter. We
  study a reduction procedure for binary trees which offers a new interpretation for the
  register function as the maximal number of reductions that can be applied to a given
  tree. In particular, the precise asymptotic behavior of the number of certain
  substructures (``branches'') that occur when reducing a tree repeatedly is determined.

  In the same manner we introduce a reduction for simple two-dimensional lattice
  paths from which a complexity measure similar to the register function can be
  derived. We analyze this quantity, as well as the (cumulative) size of an (iteratively)
  reduced lattice path asymptotically.
\end{abstract}

\section{Introduction}
\label{sec:introduction}

The aim of this paper is to investigate local substructures that appear within discrete
objects after reducing according to, in some sense, intrinsic rules. In particular, there
are two reductions we focus on: a reduction for binary trees, as well as a reduction
for simple two-dimensional lattice paths.

In order to give a summary of our results we will briefly sketch both reductions and
explain the nature of the local structures emerging when applying the reduction
repeatedly.

On a general note, we made heavy use of the open-source mathematics software system
SageMath~\cite{SageMath:2016:7.4} in order to perform the computationally intensive parts of the
asymptotic analysis for all of the parameters investigated in this paper. Files
containing these computations as well as instructions on how to run them with
SageMath can be found at \url{https://benjamin-hackl.at/publications/register-reduction/}.

\subsection{Binary Trees}
Binary trees are either a leaf or a root together with a left and a right subtree which
are binary trees. This recursive definition can be written as a symbolic equation
($\square$ and \tikz\node[circle, draw, inner sep=2.5pt] {}; mark leaves and inner nodes,
respectively):
\begin{center}\hspace*{-2cm}
\begin{tikzpicture}
[xshift=20pt,yshift=13pt,scale=0.6,level distance=15mm,
level 1/.style={sibling distance=13mm},
]
${\mathcal{B}=\square\hspace*{0.2cm} + \hspace*{0.1cm}}$
\node[draw,circle] {}
child {node {${\mathcal{B}}$}} 
child{ node {${\mathcal{B}}$}};
\end{tikzpicture}
\end{center}
By using the symbolic method (cf.~\cite[Part A]{Flajolet-Sedgewick:ta:analy}), this equation can be translated
into a functional equation for the generating function counting binary trees with respect
to their size (i.e.\ the number of inner nodes). The corresponding functional equation is
given by
\[ B(z) = 1 + z B(z)^{2}, \]
which leads to the well-known expansion
\[ B(z) = \frac{1 - \sqrt{1 - 4z}}{2z} = \sum_{n\geq 0} \frac{1}{n+1} \binom{2n}{n}
  z^{n}. \]

This means that the number of binary trees with $n$ inner nodes is given by the $n$th
Catalan number $C_{n} = \frac{1}{n+1} \binom{2n}{n}$.

By simple algebraic manipulations, it is easy to verify that the generating function
$B(z)$ satisfies the identity
\[ B(z) = 1 + \frac{z}{1 - 2z} B\Big(\frac{z^{2}}{(1 - 2z)^{2}}\Big).  \]
However, as we will see in Section~\ref{sec:reduction-register}, we
can justify this identity from a combinatorial point of view as well, and the most
important part of this combinatorial interpretation is a reduction procedure for 
binary trees.

Essentially, this procedure first removes all leaves from the tree and then ``repairs''
the resulting object by collapsing chains of nodes with only one child into one
node. More details on this reduction are provided in Section~\ref{sec:reduction-register}.

With the help of this reduction we can assign labels to all nodes in a given tree by
tracking how many iterated tree reductions it takes until the node is deleted. Note that
collapsing some nodes into one node does not count as deleting the node. In
Section~\ref{sec:reduction-register} we prove that these labels are intimately linked with
a very well-known and well-studied branching complexity measure of binary trees: the
register function.

The local structures we are interested in also become visible after labeling a tree as
described above: the so-called $r$-branches of a binary tree are the connected subgraphs
of nodes with label $r$. The number of these $r$-branches in a random tree of size $n$ is
modeled by the random variable $X_{n;r}\colon \mathcal{B}_{n} \to \N_{0}$, where
$\mathcal{B}_{n}$ is the set of all binary trees of size $n$. A proper definition as well
as results on $r$-branches can be found in Section~\ref{sec:r-branches}.

In the context of binary tree reductions we are interested in precise analyses of the
random variables $X_{n;r}$ as well as $X_{n} := \sum_{r\geq 0} X_{n;r}$, which models the
total number of branches in a random tree of size $n$. This quantity is investigated
closely in Section~\ref{sec:all-branches}.

Table~\ref{tab:trees:results} gives an overview of the results of our investigation. The
results for the register function are well-known, which is why we refer to external
literature instead. Additionally, Theorem~\ref{thm:r-branches-limit} proves asymptotic
normality for the number of $r$-branches $X_{n;r}$.

\begin{table}[ht]
  \centering
  \begin{tabular}{|c||ccc|}
    \hline
    & Register function & $r$-branches ($X_{n;r}$) & branches total ($X_{n}$) \\
    \hline\hline
    range & \cite[Sec.~1.1, Sec~2]{Flajolet-Raoult-Vuillemin:1979:register} & Proposition~\ref{prop:r-branches-range} & Proposition~\ref{prop:total-branches-bounds} \\
    explicit formula & \cite[Theorem 1]{Flajolet-Raoult-Vuillemin:1979:register} & Proposition~\ref{prop:explicit-exp-r-branch} & Proposition~\ref{cor:explicit-exp-branches} \\
    asymptotic formula & \cite[Theorem 3]{Flajolet-Raoult-Vuillemin:1979:register} & Theorem~\ref{thm:r-branches} & Theorem~\ref{thm:asy-branches}   \\ \hline
  \end{tabular}
  \vspace{1ex}
  \caption{Results: binary trees}
  \label{tab:trees:results}
\end{table}

\subsection{Lattice Paths}

Let $\mathcal{L}$ be the combinatorial class of simple two-dimensional lattice paths,
i.e., the set of all nonempty sequences over $\{\upa,\righta, \downa, \lefta\}$. It is
easy to see that the corresponding generating function is
\[ L(z) = \frac{4z}{1-4z}.  \]
Similarly to before, it is easy to check by algebraic manipulation that $L(z)$ satisfies the
functional equation
\[ L(z) = 4L\Big(\frac{z^{2}}{(1-2z)^{2}}\Big) + 4z.  \]
However, as in the case of binary trees, we will see in Section~\ref{sec:paths} that the
combinatorial interpretation of this equation is much more fruitful and gives rise to a
reduction procedure for lattice paths.

In this case it takes a bit more to fully describe the reduction. The core idea is to
reduce a given path by collapsing an entire horizontal-vertical segment (i.e.\ a path
segment that consists of a sequence of horizontal movements followed by a sequence of
vertical movements) into a single step.

The first parameter of interest in this context is the reduction degree of a
random path of length $n$, which is the number of repeated reductions that it takes until
the entire path is reduced to a single step. We will model this parameter with the random
variable $D_{n}\colon \mathcal{L}_{n}\to \N_{0}$, where $\mathcal{L}_{n}$ consists of all
simple two-dimensional lattice paths of length $n$.

As an analogue to the number of $r$-branches in a given binary tree we consider the length
of the $r$th fringe, i.e., the $r$th reduction of a given lattice path. This quantity is
modeled by the random variable $X_{n;r}^{L}\colon \mathcal{L}\to \N_{0}$.

By summation of the length of the $r$th fringe for $r\geq 0$ we obtain the total fringe
size $X_{n}^{L} := \sum_{r\geq 0} X_{n;r}^{L}$. In some sense, the total fringe size
measures the complexity of horizontal-vertical direction changes of a given lattice
path. Both, the $r$th fringe size as well as the total fringe size are analyzed in
Section~\ref{sec:fringes}.

Table~\ref{tab:paths:results} gives an overview of the results of our investigation.

\begin{table}[ht]
  \centering
  \begin{tabular}{|c||ccc|}
    \hline
    & Reduction degree ($D_{n}$) & $r$-fringes ($X_{n;r}^{L}$) & total fringe size ($X_{n}^{L}$) \\
    \hline\hline
    range & Proposition~\ref{prop:paths:rdeg-range} & Proposition~\ref{prop:r-fringes-range} & Proposition~\ref{prop:paths:total-fringe-range} \\
    explicit formula & Corollary~\ref{cor:L-asy-expansion} & Proposition~\ref{prop:fringes-explicit} & Corollary~\ref{cor:fringe-all-explicit} \\
    asymptotic formula & Theorem~\ref{thm:lat-results} & Theorem~\ref{thm:fringe-sizes} & Theorem~\ref{thm:fringe-size} \\ \hline
  \end{tabular}
  \vspace{1ex}
  \caption{Results: lattice paths}
  \label{tab:paths:results}
\end{table}

\section{Tree Reductions and the Register Function}
\label{sec:reduction-register}

\subsection{Motivation and Preliminaries}
As mentioned in the introduction, we want to find a combinatorial proof of the
following proposition.

\begin{proposition}\label{prop:tree-identity}
  The generating function counting binary trees by the number of inner nodes, $B(z) =
  \frac{1 - \sqrt{1 - 4z}}{2z}$, satisfies the identity
  \begin{equation}\label{eq:u1}
    B(z)=1+\frac{z}{1-2z}B\Big(\frac{z^2}{(1-2z)^2}\Big).
  \end{equation}
\end{proposition}

\begin{figure}[ht]
  \centering
  \includegraphics[scale=1]{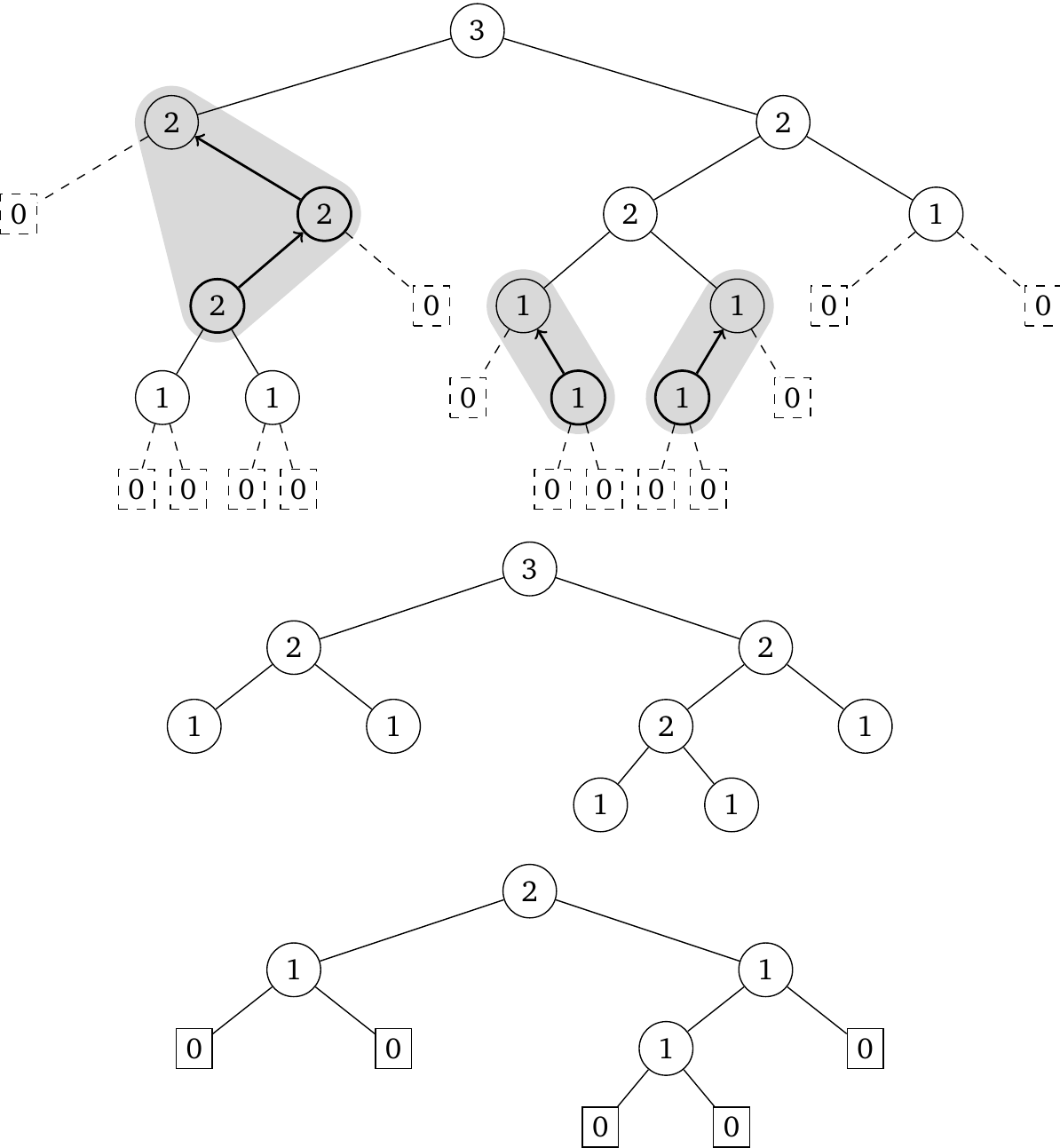}
  \caption{Illustration of the tree reduction $\Phi$: in the first tree, the leaves are
    deleted (dashed nodes) and nodes with exactly one child are merged (gray overlay). The
  second tree shows the result of these operations. Finally, in the last tree all nodes
  without children are marked as leaves.}
  \label{fig:trimming}
\end{figure}

\begin{proof}
We consider the following reduction of a binary tree $t$, which we write as $\Phi(t)$:

First, all leaves of $t$ are erased. Then, if a node has only one
child, these two nodes are merged; this operation will be repeated as long as
there are such nodes. The leaves of the reduced tree are precisely the nodes without
children.

This operation was introduced in \cite{YaYa10}. The various steps of the reduction are
depicted in Figure~\ref{fig:trimming}. The numbers attached to the nodes
will be explained later.

Note that $\Phi(\square)$ is undefined, so this is a partial function.
Of course, many different trees are mapped to the same binary tree. However, they can all
be obtained from a given reduced tree by the following operations: 

All leaves and all internal nodes in the tree are replaced by chains
of internal nodes. In such a chain, there has to be at least one leaf attached to every
internal node; the symbolic equation for chains is 
\begin{center}\hspace{-4cm}
  \begin{tikzpicture}[xshift=20pt,yshift=13pt,scale=0.6,level distance=15mm,
    level 1/.style={sibling distance=13mm}]
    ${\mathcal{C} = \hspace*{0.1cm}}$ 
    \node[draw,circle] {}
    child {node[draw, rectangle] {}} 
    child{ node[draw, rectangle] {}};
    $\hspace*{1.3cm} + $
    \node[draw,circle] {}
    child {node[draw, rectangle] {}} 
    child{ node {${\mathcal{C}}$}};
    $\hspace*{1.3cm} + $
    \node[draw,circle] {}
    child {node {${\mathcal{C}}$}} 
    child{ node[draw, rectangle] {}};
    $\hspace*{1.3cm}.$
  \end{tikzpicture}
\end{center}
Obviously, these substitutions do not only restore the (previously deleted)
leaves, but can also ``unmerge'' previously merged nodes. Thus, all trees that reduce to
some tree $t'$ can be reconstructed from $t'$.

From the symbolic equation of chains above, we find that the generating function $C(z)$
counting chains with respect to their size (i.e.\ number of internal nodes) satisfies the
equation $C(z) = z + 2z C(z)$ and thus, we obtain
\[ C(z) = \frac{z}{1 - 2z}. \]

Finally, if $F(z)$ is a generating function counting some family of binary trees, then the
bivariate generating function $vF(zv)$ counts the same family with respect to size
(variable $z$) and number of leaves (variable $v$). This is a direct consequence of the
fact that binary trees with $n$ inner nodes have $n+1$ leaves. 

Therefore, replacing all nodes of a
binary tree with chains corresponds to the substitutions $v\mapsto \frac{z}{1 - 2z}$ and
$z\mapsto \frac{z}{1 - 2z}$ in the language of generating functions. Therefore, all binary
trees that can be reconstructed from a reduced version of itself are counted by
\[ \frac{z}{1 - 2z} B\Big(\frac{z^{2}}{(1 - 2z)^{2}}\Big). \]
By all these considerations, \eqref{eq:u1} can be interpreted combinatorially as the
following statement: a binary tree is either just $\square$, or it can be reconstructed
from another binary tree where all nodes are replaced by chains.
\end{proof}

\begin{remark}
  Note that \eqref{eq:u1} can be used to find a very simple proof for a well-known
  identity for Catalan numbers: 

  Comparing the coefficients of $z^{n+1}$, \eqref{eq:u1} leads to
  \begin{align*}
    C_{n+1}&=[z^{n+1}]\sum_{k\ge0}C_k\frac{z^{2k+1}}{(1-2z)^{2k+1}}=
             \sum_{k\ge0}C_k[z^{n-2k}]\sum_{j\ge0}2^{j}\binom{2k+j}{j}z^j\\
           &=\sum_{0\le k\le n/2}C_k2^{n-2k}\binom{n}{2k},
  \end{align*}
which is known as Touchard's identity~\cite{shapiro-1976, touchard-1924}.
\end{remark}

With this interpretation in mind, \eqref{eq:u1} can also be seen as a recursive process to
generate binary trees by repeated substitution of chains. This process can be modeled by
the generating functions
\begin{equation}\label{eq:register-leq-recursion}
B_0(z)=1, \quad B_r(z)=1+\frac{z}{1-2z}B_{r-1}\Big(\frac{z^2}{(1-2z)^2}\Big), \quad r\ge1.
\end{equation}
By construction, $B_{r}(z)$ is the generating function of all binary trees that can be
constructed from $\square$ with up to $r$ expansions---or, equivalently---all binary trees
that can be reduced to $\square$ by applying $\Phi$ up to $r$ times.

Expanding the first few functions gives
\begin{align*}
 B_1(z)&=1+z+2{z}^{2}+4{z}^{3}+8{z}^{4}+16{z}^{5}+32{z}^{6}+64{z}^{7}+128{z}^{8}+256{z}^{9}+512{z}^{10}+\cdots,\\
 B_2(z)&=1+z+2{z}^{2}+5{z}^{3}+14{z}^{4}+42{z}^{5}+132{z}^{6}+428{z}^{7}+1416{z}^{8}+4744{z}^{9}+\cdots,\\
 B_3(z)&=1+z+2{z}^{2}+5{z}^{3}+14{z}^{4}+42{z}^{5}+132{z}^{6}+429{z}^{7}+1430{z}^{8}+4862{z}^{9}+\cdots.
\end{align*}

As it turns out, these generating functions are inherently linked with the \emph{register
  function} (also known as the Horton-Strahler number) of binary trees. In order to
understand this connection, we introduce the register function and prove a simple property
regarding the tree reduction $\Phi$.

The register function is recursively defined: for the binary tree consisting of
only a leaf we have $\Reg(\square) = 0$, and if a binary
tree $t$ has subtrees $t_{1}$ and $t_{2}$, then the register function is defined to be
\[ \Reg(t) = \begin{cases}
    \max\{\Reg(t_{1}), \Reg(t_{2})\} & \text{ for } \Reg(t_{1}) \neq \Reg(t_{2}),\\
    \Reg(t_{1}) + 1 & \text{ otherwise.}
  \end{cases}\]
In particular, the numbers attached to the nodes in Figure~\ref{fig:trimming} represent
the values of the register function of the subtree rooted at the respective node.

Historically, the idea of the register function originated (as the Horton-Strahler
numbers) in \cite{Horton45, Strahler52} in the study of the complexity of river
networks. However, the very same concept also occurs within a computer science
context: arithmetic expressions with binary operators can be expressed as a binary tree
with data in the leaves and operators in the internal nodes. Then, the register function of this
binary expression tree corresponds to the minimal number of registers needed to evaluate
the expression.

There are several publications in which the register function and related concepts are
investigated in great detail, for example Flajolet, Raoult, and Vuillemin~\cite{Flajolet-Raoult-Vuillemin:1979:register},
Kemp~\cite{Kemp79}, Flajolet and Prodinger~\cite{Flajolet-Prodinger:1986:regis},
Nebel~\cite{Nebel:2002:horton-strahler-unified}, Drmota and Prodinger~\cite{DrPr06}, and
Viennot~\cite{Viennot:2002:strah-dyck}. For a detailed survey on the register
function and related topics see~\cite{prodinger:register-survey}.

We continue by observing that the tree reduction $\Phi$ is a very natural operation
regarding the register function:

\begin{proposition}\label{prop:reg-compact}
  Let $t$ be a binary tree with $\Reg(t) = r \geq 1$. Then $\Phi(t)$ is well-defined and
  the register function of the reduced tree is $\Reg(\Phi(t)) = r - 1$.
\end{proposition}
\begin{proof}
  First, observe that all trees with at least one internal node have a node with two leaves
  attached. Therefore, this node has register function $1$---and
  thus, only $\square$ has register function $0$. Consequently, if we have $\Reg(t) \geq 1$,
  $t$ cannot be $\square$, meaning that $\Phi(t)$ is well-defined.

  Now take an arbitrary binary tree $t$ with at least one internal node and assume that we
  have $\Reg(\Phi(t)) = r$. As described above, the tree $t$ can be reconstructed from
  $\Phi(t)$ by replacing all nodes (i.e.\ leaves and internal nodes) by chains of internal
  nodes.

  When replacing internal nodes with chains of internal nodes, nothing changes for the
  register function: the value is just propagated up along the chain. However, if all
  leaves are replaced by chains, the register function of all subtrees that are rooted at a
  internal node increases by $1$, resulting in $\Reg(t) = r+1$. This proves the proposition.
\end{proof}

As an immediate consequence of Proposition~\ref{prop:reg-compact} we find that $\Phi$ can
be applied $r$ times repeatedly to some binary tree $t$ if and only if $\Reg(t) \geq r$
holds. In particular, we obtain
\begin{equation}
  \label{eq:reg-characterization}
  \Phi^{r}(t) = \square \quad \iff \quad \Reg(t) = r.
\end{equation}
With~\eqref{eq:reg-characterization}, the link between the generating functions $B_{r}(z)$
from above and the register function becomes clear: $B_{r}(z)$ is exactly the generating
function of binary trees with register function $\leq r$. 

In order to analyze these recursively defined generating functions an explicit
representation is convenient. As it turns out, the substitution $z = \frac{u}{(1 +
  u)^{2}} =: Z(u)$ is a helpful tool in this context.

\begin{proposition}\label{prop:subs-prop}
  Consider the complex functions
  \begin{align*}
    Z(u) &= \frac{u}{(1+u)^{2}} & \text{for } u\in \C\setminus\{-1\},\\
    U(z) &= \frac{1 - \sqrt{1 - 4z}}{2z} - 1 & \text{for } z\in \C,
  \end{align*}
  where the principal branch of the square root function is chosen as usual, i.e., as a
  holomorphic function on $\C\setminus \R_{\leq 0}$ such that $\sqrt{1} = 1$. Then the
  following properties hold:
  \begin{enumerate}[(a)]
  \item Let $\mathcal{Z} = \C \setminus [1/4, \infty)$ and $\mathcal{U} = \{u\in \C \mid |u| < 1\}$. Then
    $U\colon \mathcal{Z} \to \mathcal{U}$ and $Z\colon \mathcal{U}\to \mathcal{Z}$ are bijective holomorphic functions
    which are inverses of each other.
  \item Let $\overline{\mathcal{U}} = \mathcal{U} \cup \{\exp(-t \pi i) \mid 0 \le t <
    1\}$. Then $U\colon \C\setminus\{-1\}\to \overline{\mathcal{U}}$ is bijective with inverse $Z$.
  \item\label{prop:subs-prop:relations} The relations
    \[ Z'(u) = \frac{1-u}{(1+u)^{3}}\quad\text{ and }\quad \frac{Z(u)}{1 - 2Z(u)} =
      \frac{u}{1+u^{2}}  \]
    hold for $u\in \C\setminus\{-1\}$
  \item\label{prop:subs-prop:diagram} For the function $\sigma\colon \C\setminus\{\frac{1}{2}\} \to \C$ with
    $\sigma(z) = \frac{z^{2}}{(1 - 2z)^{2}}$, the diagram
    \begin{center}
      \begin{tikzpicture}
        \node (A) at (0,0) {$\mathcal{U}$};
        \node (B) at (0,2) {$\mathcal{Z}$};
        \node (C) at (3.5,0) {$\mathcal{U}$};
        \node (D) at (3.5,2) {$\mathcal{Z}$};
        \draw[->] (A) to[->] node[left]{$Z$} (B);
        \draw[->] (A) to[->] node[below]{$u\mapsto u^{2}$} (C);
        \draw[->] (B) to[->] node[above]{$\sigma$} (D);
        \draw[->] (C) to[->] node[right]{$Z$} (D);
      \end{tikzpicture}
    \end{center}
    commutes, i.e. we have $\sigma \circ Z = Z \circ (u\mapsto u^{2})$.
  \item\label{prop:subs-prop:pole-translation} Let $\alpha\in\C\setminus\{0, -1\}$,
    $u\in\C\setminus\{\alpha, 1/\alpha\}$ and $z=Z(u)$. Then
    \begin{equation*}
      \frac{u}{(u-\alpha)(u-\frac1\alpha)}=-\frac{zZ(\alpha)}{z-Z(\alpha)}.
    \end{equation*}
    For $\alpha = -1$ we find $\frac{u}{(1+u)^{2}} = Z(u) = z$.
  \end{enumerate}
\end{proposition}
\begin{proof}\ 
  \begin{enumerate}[(a)]
  \item We first note that $Z$ is well-defined and holomorphic on $\mathcal{U}$ with
    $Z'(u)\neq 0$ for all $u\in \mathcal{U}$. If $|u| = 1$, then
    \[ Z(u) = \frac{1}{u + \frac{1}{u} + 2} = \frac{1}{2 + 2\Re u}.  \]
    Thus, the image of the unit circle without $u = -1$ is the interval $[1/4, \infty)$.

    For every $z\in \C\setminus\{0\}$, $z = Z(u)$ is equivalent to
    \begin{equation}
      \label{eq:substitution-equation}
      u^{2} + u \Big(2 - \frac{1}{z}\Big) + 1 = 0
    \end{equation}
    which has two not necessarily distinct solutions $u_{1}$, $u_{2}\in \C$ with $u_{1}
    u_{2} = 1$. W.l.o.g., $|u_{1}| \leq |u_{2}|$. Thus either $u_{1} \in \mathcal{U}$ and
    $|u_{2}| > 1$ or $|u_{1}| = |u_{2}| = 1$. In the latter case, we have $z \in [1/4,
    \infty)$. For $z = 0$, $z = Z(u)$ is equivalent to $u = 0$. This implies that
    $Z\colon \mathcal{U} \to \mathcal{Z}$ is bijective.
    
    Furthermore, $Z\colon \mathcal{U}\to \mathcal{Z}$ has a holomorphic inverse $Z^{-1}$ defined on
    the simply connected region $\mathcal{Z}$. Solving~\eqref{eq:substitution-equation} explicitly
    yields
    \[ u = \frac{1 \pm \sqrt{1 - 4z}}{2z} - 1.  \]
    In a neighborhood of zero, we must have $Z^{-1}(z) = U(z)$, because
    \[ \frac{1 + \sqrt{1 - 4z}}{2z} - 1  \]
    has a pole at $z = 0$. Altogether this proves that $U$ is the inverse of
    $Z$.
  \item For $z\in [1/4, \infty)$, we know that $U(z)$ is on the unit circle. It is easily
    checked that $\Im U(z) = -\sqrt{\abs{1 - 4z}}/(2z)$ for these $z$, thus $\Im U(z) \in
    \overline{\mathcal{U}}$.

  \item The two relations follow directly from the definition of $Z$.

  \item This can be shown by straightforward computation: we obtain
    \[ \sigma(Z(u)) = \Big(\frac{Z(u)}{1 - 2Z(u)}\Big)^{2} = \Big(\frac{u}{1 +
      u^{2}}\Big)^{2} = Z(u^{2}),  \]
  where~(\ref{prop:subs-prop:relations}) is used.
\item By writing~\eqref{eq:substitution-equation} as
    \begin{equation*}
      u+\frac 1u = \frac1z -2,
    \end{equation*}
    we have
    \begin{align*}
      \frac{u}{(u-\alpha)(u-\frac1\alpha)} &= -\frac{\alpha}{(u-\alpha)(\frac1u
        -
        \alpha)}=-\frac{\alpha}{1-\alpha(u+\frac1u)+\alpha^2}\\&=-\frac{\alpha}{(\alpha+1)^2-\frac{\alpha}{z}}
      = -\frac{zZ(\alpha)}{z-Z(\alpha)}.
    \end{align*}
\qedhere
  \end{enumerate}
\end{proof}

In a nutshell, the fact that $\sigma \circ Z = Z \circ (u\mapsto
u^{2})$ means that applying $\sigma$ in the ``$z$-world'' corresponds to squaring in the
``$u$-world''. As we will see in a moment, this is very useful for expressing recursively
defined generating functions like the one encountered above explicitly.

\begin{proposition}\label{prop:recursion-solution}
  Let $F_{0}$, $D$, and $E$ be complex functions that are analytic in a neighborhood of
  $0$. Then the recursively defined functions
  \begin{equation}\label{eq:recursion-statement} 
    F_{r}(z) = D(z) + E(z) F_{r-1}(\sigma(z)),\quad r\geq 1
  \end{equation}
  can be written explicitly by means of the substitution $z = \frac{u}{(1 + u)^{2}}$ as
  \begin{equation}\label{eq:recursion-solution}
    F_{r}(z) = \sum_{j=0}^{r-1} D\bigg(\frac{u^{2^{j}}}{(1 + u^{2^{j}})^{2}}\bigg)
    \prod_{k=0}^{j-1} E\bigg(\frac{u^{2^{k}}}{(1 + u^{2^{k}})^{2}}\bigg) +
    F_{0}\bigg(\frac{u^{2^{r}}}{(1 + u^{2^{r}})^{2}}\bigg) \prod_{k=0}^{r-1}
    E\bigg(\frac{u^{2^{k}}}{(1 + u^{2^{k}})^{2}}\bigg).
  \end{equation}
\end{proposition}
\begin{proof}
Let $j\in \N$. Observe that by repeated application of
Property~(\ref{prop:subs-prop:diagram}) of Proposition~\ref{prop:subs-prop} we can write
\[ \frac{u^{2^{j}}}{(1 + u^{2^{j}})^{2}} = Z(u^{2^{j}}) = \sigma(Z(u^{2^{j-1}}))
  = \cdots = \sigma^{j}(Z(u)) = \sigma^{j}(z),  \]
where $\sigma^{j}(z)$ denotes the $j$-fold application of $\sigma$ to $z$. This lets us
write~\eqref{eq:recursion-solution} as
\[ F_{r}(z) = \sum_{j=0}^{r-1} D(\sigma^{j}(z))\prod_{k=0}^{j-1} E(\sigma^{k}(z)) +
  F_{0}(\sigma^{r}(z)) \prod_{k=0}^{r-1} E(\sigma^{k}(z)).  \]
This expression follows from~\eqref{eq:recursion-statement} by induction over $r$.
\end{proof}

With Proposition~\ref{prop:recursion-solution} we have an appropriate tool for analyzing
$B_{r}(z)$, the generating function enumerating binary trees
with register function $\leq r$. With $D(z) = 1$, $E(z) = \frac{z}{1 - 2z}$, and 
Property~(\ref{prop:subs-prop:relations}) of Proposition~\ref{prop:subs-prop} the
recurrence in \eqref{eq:register-leq-recursion} yields
\begin{equation}\label{eq:register-leq-explicit}
  B_{r}(z) = \frac{1 - u^{2}}{u} \sum_{j=0}^{r} \frac{u^{2^{j}}}{1 - u^{2^{j+1}}}.
\end{equation}

Note that at this point, we can determine the generating function $B_{r}^{=}(z)$ counting
binary trees with register function equal to $r$ with respect to their size as
\begin{equation}\label{eq:trees:equal-ogf}
  B_{r}^{=}(z) = B_{r}(z) - B_{r-1}(z) = \frac{1-u^{2}}{u} \frac{u^{2^{r}}}{1 -
    u^{2^{r+1}}}.
\end{equation}
This explicit representation of $B_{r}^{=}(z)$ could be used to determine the asymptotic
behavior of the register function. However, as these properties are well-known
(cf.~\cite{Flajolet-Raoult-Vuillemin:1979:register}), we will continue in a different direction by studying the number of
so-called $r$-branches---where we will also encounter the generating function
$B_{r}^{=}(z)$ again.

\subsection{$r$-branches}
\label{sec:r-branches}

The register function associates a value to each node (internal nodes as well as leaves), and the
value at the root is the value of the register function of the tree. An $r$-branch is a
maximal chain of nodes labeled $r$. This must be a chain, since the merging of two such
chains would already result in the higher value $r+1$. The nodes of the tree are
partitioned into such chains, from $r=0,1,\ldots$. Figure~\ref{fig:reg-branches}
illustrates this situation for a tree of size $13$.

The goal of this section is the study of the parameter ``number of $r$-branches'', in
particular, the average number of them, assuming that all binary trees of size $n$ are
equally likely.

Formally, we investigate this parameter via the family of random variables
$(X_{n;r})_{\substack{n\geq 0\\ r\geq 0}}$ where $X_{n;r}\colon \mathcal{B}_{n} \to
\N_{0}$ counts the number of $r$-branches in binary trees of size $n$.

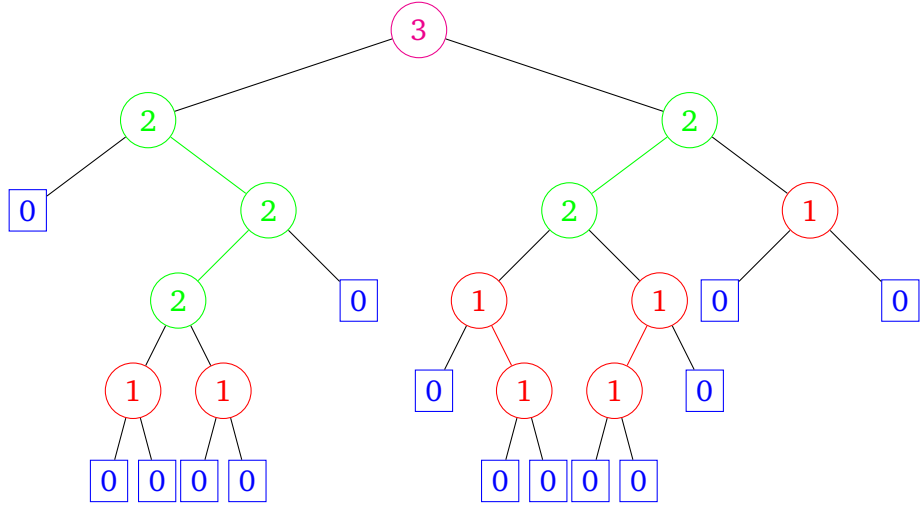
\begin{figure}[ht]
  \centering
  \begin{tikzpicture}[scale=0.8,level distance=15mm,
    level 1/.style={sibling distance=90mm},
    level 2/.style={sibling distance=40mm},
    level 3/.style={sibling distance=30mm},
    level 4/.style={sibling distance=15mm},
    level 5/.style={sibling distance=8mm},
    ]
    \node[circle,draw,color=magenta] {${3}$}
    child {node[circle,draw,color=green] {${2}$} 
      child[color=black]{ node [rectangle,draw,color=blue]{${0}$}}
      child[color=green] {node [circle,draw,color=green]{${2}$}
        child[color=green] {node[circle,draw,color=green] {${2}$}
          child[color=black] {node[circle,draw,color=red] {${1}$}
            child {node[rectangle, draw, color=blue] {${0}$}}
            child {node[rectangle, draw, color=blue] {${0}$}}
          }
          child[color=black] {node[circle,draw,color=red] {${1}$}
            child {node[rectangle, draw, color=blue] {${0}$}}
            child {node[rectangle, draw, color=blue] {${0}$}}
          }
        }
        child[color=black] {node [rectangle,draw,color=blue]{${0}$}
        }
      }
    }
    child {node[circle,draw,color=green] {${2}$}
      child[color=green] {node [circle,draw,color=green]{${2}$}
        child [color=black]{node[circle,draw,color=red] {${1}$}
          child[color=black] {node[rectangle,draw,color=blue] {${0}$}}
          child[color=red] {node [circle,draw,color=red]{${1}$}
            child[color=black] {node[rectangle,draw,color=blue] {${0}$}}
            child [color=black]{node[rectangle,draw,color=blue] {${0}$}}
          }
        }
        child[color=black] {node [circle,draw,color=red]{${1}$}
          child [color=red]{node [circle,draw,color=red]{${1}$}
            child [color=black]{node [rectangle,draw,color=blue]{${0}$}}
            child [color=black]{node [rectangle,draw,color=blue]{${0}$}}
          }
          child {node[rectangle,draw,color=blue] {${0}$}}
        }
      }
      child {node [circle,draw,color=red]{${1}$}
        child {node[rectangle,draw,color=blue] {${0}$}}
        child {node[rectangle,draw,color=blue] {${0}$}}
      }
    };
  \end{tikzpicture}
  \caption{Binary tree with colored $r$-branches}
  \label{fig:reg-branches}
\end{figure}

This parameter was the main object of the paper \cite{YaYa10}, and some partial results
were given that we are now going to extend. In contrast to this paper, our approach relies
heavily on generating functions which, besides allowing us to verify the results in a
relatively straightforward way, also enables us to extract explicit formul\ae{} for the
expectation (and, in principle, also for higher moments).

A parameter that was not investigated in \cite{YaYa10} is the total number of
$r$-branches, for any $r$, i.e., the sum over $r\ge0$. Here, asymptotics are trickier, and
the basic approach from \cite{YaYa10} cannot be applied. However, in this paper we use the
Mellin transform, combined with singularity analysis of generating functions, a multi-layer
approach that also allowed one of us several years ago to solve a problem by Yekutieli and
Mandelbrot, cf.~\cite{Prodinger97}. The origins of singularity analysis can be found in
\cite{Flajolet-Odlyzko:1990:singul}, and for a detailed survey see \cite{Flajolet-Sedgewick:ta:analy}.

For reasons of comparisons, let us mention that the value of register function in
\cite{YaYa10} are one higher than here, and that $n$ generally refers there to the number
of leaves, not nodes as here. 

According to our previous considerations, after $r$ iterations of $\Phi$, the $r$-branches
become leaves (or, equivalently, $0$-branches).

We begin our detailed analysis of the random variables enumerating $r$-branches by
studying sharp bounds for this parameter.

\begin{proposition}\label{prop:r-branches-range}
  Let $n$, $r\in \N_{0}$. If $r = 0$, then $X_{n;0}$ is a deterministic quantity with
  $X_{n;0} = n+1$. For $r > 0$, the bound
  \[ \llbracket n > 0 \text{ and } r = 1\rrbracket \leq X_{n;r} \leq \Big\lfloor
    \frac{n+1}{2^{r}} \Big\rfloor  \]
  holds and is sharp.
\end{proposition}
\begin{proof}
  First, recall that $r$-branches are nothing else than leaves in the $r$-fold reduced
  tree. Thus, $X_{n;0}$ counts the number of leaves in a binary tree with $n$ inner
  nodes---and it is a well-known fact that binary trees with $n$ inner nodes always have
  $n+1$ leaves.

  For the lower bound we observe that in every tree with at least one inner node, there is
  a node to which two leaves are attached. This node is part of a (possibly larger)
  $1$-branch. Therefore, $1$ is a lower bound for $X_{n;1}$ where $n > 0$. Chains are an
  example for arbitrarily large binary trees where the lower bounds $1$ and $0$ are attained
  for $r = 1$ and $r > 1$, respectively.

  As there are finitely many binary trees of size $n$, there is a tree $t$ for which
  $X_{n;r}$ attains its maximum value $M \in \N_{0}$, meaning that the $r$-fold reduced
  tree $\Phi^{r}(t)$ has $M$ leaves. In order to obtain an estimate between $M$ and $n$ we
  expand the reduced tree $r$-times by successively replacing leaves by
  cherries, which are chains of size one. By doing so, the number of leaves doubles after
  every iteration, which means that our new tree has $M\cdot 2^{r}$ leaves---or,
  equivalently, $M\cdot 2^{r} - 1$ inner nodes. Because $t$ cannot be smaller than the tree we
  have just constructed, the inequality $M \cdot 2^{r} - 1 \leq n$ has to hold. This
  proves the upper bound in the statement above.

  In order to show that the upper bound is sharp as well, we consider the family of binary
  trees $(B_{m})_{m\geq 1}$, where $B_{m}$ denotes the unique almost complete binary tree
  with $m$ leaves, which is constructed by adding the nodes layer-to-layer from left to
  right.
  \begin{figure}[ht]
    \centering
    \begin{subfigure}[b]{0.55\textwidth}
    \begin{tikzpicture}[scale=0.6,level distance=15mm,inner sep=3pt,
      level 1/.style={sibling distance=90mm},
      level 2/.style={sibling distance=38mm},
      level 3/.style={sibling distance=25mm},
      level 4/.style={dashed,sibling distance=11mm},
      ]
      \node[draw,circle] {}
      child {node[draw,circle] {}
        child {node[draw,circle] {}
          child {node[draw,rectangle] {}}
          child {node[draw,rectangle] {}}
        }
        child {node[draw,circle] {}
          child {node[draw,rectangle] {}}
          child {node[draw,rectangle] {}}
        }
      }
      child {node[draw,circle] {}
        child {node[draw,rectangle] {}}
        child {node[draw,rectangle] {}}
      };
    \end{tikzpicture}
    \caption{$B_{6}$}
    \end{subfigure}
    \qquad
    \begin{subfigure}[b]{0.25\textwidth}
    \begin{tikzpicture}[scale=0.6,level distance=15mm,inner sep=3pt,
      level 1/.style={sibling distance=38mm},
      level 2/.style={sibling distance=25mm},
      level 3/.style={dashed,sibling distance=11mm},
      ]
      \node[draw,circle] {}
      child {node[draw,circle] {}
        child {node[draw,rectangle] {}
        }
        child {node[draw,rectangle] {}
      }}
      child {node[draw,rectangle] {}
      };
    \end{tikzpicture}
    \caption{$B_{3}$}
    \label{fig:almostcomplete:B3}
    \end{subfigure}

    \caption{Almost complete binary trees}
    \label{fig:almostcomplete}
  \end{figure}
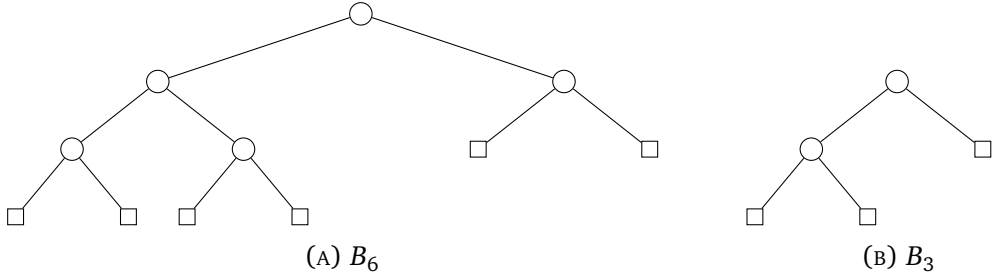
  For these trees, we can prove that $\Phi(B_{m}) = B_{\lfloor m/2 \rfloor}$: in case $m$
  is even, reducing the tree is equivalent to replacing all cherries on the lower
  levels by leaves, effectively halving the number of leaves. If $m = 2k+1$ is odd, there
  is a node whose left and right child is an inner node and a leaf, respectively. In
  particular, the subtree in question looks like $B_{3}$ illustrated in
  Figure~\ref{fig:almostcomplete:B3}. When reducing this tree, the left child has to be
  merged with its parent. This shows that in total, $\Phi(B_{2k+1})$ has $k$ leaves.

  By applying $\Phi(B_{m}) = B_{\lfloor m/2 \rfloor}$ iteratively, and by
  \[ \Big\lfloor \frac{m}{2^{r}}\Big\rfloor  = \bigg\lfloor \frac{1}{2} \Big\lfloor
    \frac{m}{2^{r-1}}\Big\rfloor \bigg\rfloor \]
  we see that $B_{n+1}$, which is a binary tree of size $n$, attains the upper bound for
  the number of $r$-branches.
\end{proof}

Next we analyze the asymptotic behavior of the expectation and variance of $X_{n;r}$.

\begin{theorem}\label{thm:r-branches} 
  Let $r\in \N_{0}$ be fixed. The expected number $\E X_{n;r}$ of $r$-branches in binary trees of size
  $n$ and the corresponding variance $\V X_{n;r}$ have the asymptotic expansions
  \begin{align}
    \label{eq:asy-exp-r-branch}
    \E X_{n;r} & = \frac{n}{4^{r}} + \frac{1}{6} \Big(1 +
    \frac{5}{4^{r}}\Big) + \frac{1}{20 n}\Big(4^{r} - \frac{1}{4^{r}}\Big) + \frac{1}{12
    n^{2}} \Big(\frac{5\cdot 16^{r}}{21}  - \frac{7\cdot 4^{r}}{10}  + \frac{97}{210\cdot 4^{r}}\Big)
    + O(n^{-3}), \\
    \label{eq:asy-var-r-branch}
    \V X_{n;r} & = \frac{4^r-1}{3\cdot 16^r} n - \frac{2\cdot 16^{r} - 25\cdot 4^{r} + 23}{90\cdot
    16^{r}} - \frac{13\cdot 64^{r} - 14\cdot 16^{r} + 7\cdot 4^{r} - 6}{420\cdot 16^{r} n} + O(n^{-2}).
  \end{align}
\end{theorem}
\begin{remark}
  The main terms (without error terms) of the asymptotic expansions for the expectation
  and the variance of the number of $r$-branches have already been determined
  in~\cite{Moon:1980:random-channel}.
\end{remark}
\begin{proof}
We begin our asymptotic analysis by constructing the generating function of the total
number of leaves in all trees of size $n$. First, observe that the bivariate generating
function allowing us to count the leaves of the binary trees is $vB(zv)$. Hence, the
generating function counting the total number of leaves among all trees of size $n$ is
given by 
\begin{equation*}
\frac{\partial}{\partial v}vB(zv)\Big|_{v=1}=\frac1{\sqrt{1-4z}}=\frac{1+u}{1-u}.
\end{equation*}
Following the same recursive procedure as described in the proof of
Proposition~\ref{prop:tree-identity} and replacing all nodes of a given tree by chains, 
the leaves become $1$-branches. Generally speaking, expanding a tree lets the $r$-branches
become $(r+1)$-branches. In particular, this means that after $r$ iterations of the tree expansion, the
leaves have become $r$-branches.

With this in mind, we want to construct the generating function $F_{r}^{(1)}(z)$ that
enumerates the sum of the number of $r$-branches over all trees with the same size, which
is marked by $z$. As $0$-branches are leaves, the expression determined above is precisely
$F_{0}^{(1)}(z)$. Applying the tree expansion operator $r$-times to $F_{0}^{(1)}(z)$
yields $F_{r}^{(1)}(z)$. This is justified by the argument
\begin{align*}
  F_{r}^{(1)}(z) = \sum_{t\in \mathcal{B}} \#(r\text{-branches of } t) z^{|t|} & = \sum_{t'\in \mathcal{B}}
  \sum_{\substack{t\in \mathcal{B} \\ \Phi(t) = t' }} \#(r\text{-branches of } t) z^{|t|}
  \\ & = \sum_{t'\in \mathcal{B}} \#((r-1)\text{-branches of } t') \sum_{\substack{t\in
      \mathcal{B} \\ \Phi(t) = t'}} z^{|t|}\\
  & = \frac{z}{1-2z} F_{r-1}^{(1)} \Big(\frac{z^{2}}{(1-2z)^{2}}\Big),
\end{align*}
where $|t|$ denotes the size of a tree $t \in \mathcal{B}$.

Altogether, we obtain the recursion
\begin{equation*}
F_{0}^{(1)}(z)=\frac1{\sqrt{1-4z}},\quad F_{r}^{(1)}(z)=\frac{z}{1-2z}F_{r-1}^{(1)}\Big(\frac{z^2}{(1-2z)^2}\Big),\quad r\ge1.
\end{equation*}
By construction, dividing the $n$th coefficient of $F_{r}^{(1)}(z)$ by $C_{n}$ yields
\begin{equation*}
\E X_{n;r} = \frac1{C_n}[z^n]F_{r}^{(1)}(z),
\end{equation*}
which is the expected number of $r$-branches in a random tree of size $n$.

In order to analyze $F_{r}^{(1)}(z)$ we rewrite it using
Proposition~\ref{prop:recursion-solution} and the fact that $D(z) = 0$ and $E(z) =
\frac{z}{1 - 2z}$. Thus, we obtain
\begin{equation}\label{eq:trees:r-branch-expectation-gf} 
F_{r}^{(1)}(z) = \frac{1 - u^{2}}{u} \frac{u^{2^{r}}}{(1 - u^{2^{r}})^{2}}. 
\end{equation}
The generating function $F_{r}^{(1)}(z)$ has a singularity at $z=1/4$, so we have to locally
expand the function in terms of $\sqrt{1 - 4z}$ such that the methods of singularity
analysis can be applied.

Expansion yields
\begin{align*}
F_{r}^{(1)}(z) = \frac1{{4}^{r}\sqrt{1-4z}}+ \frac13({4}^{-r}-1 ) \sqrt{1-4z}+
  \frac1{15}({4}^{1-r}-5+{4}^{r}) ({1-4z})^{3/2} + O((1 - 4z)^{5/2}).
\end{align*}
Singularity analysis~\cite[Chapter VI]{Flajolet-Sedgewick:ta:analy} guarantees that one can read off
coefficients in this expansion:
\begin{equation*}
[z^{n}] F_{r}^{(1)}(z) = \frac{4^{n}}{\sqrt{\pi}} \bigg( \frac{1}{4^{r} \sqrt{n}} +
\frac{1}{6n^{3/2}}\Big(1 - \frac{7}{4^{r+1}}\Big) + \frac{1}{n^{5/2}}\Big(\frac{4^{r}}{20} - \frac{3}{16} +
\frac{93}{640\cdot 4^{r}}\Big)  + O(n^{-7/2}) \bigg).
\end{equation*}
The asymptotics of $C_n$ are straightforward, especially for a computer. By performing
singularity analysis on the generating function $B(z)$ we obtain
\[ C_{n} = \frac{4^{n}}{\sqrt{\pi}} \bigg( \frac{1}{n^{3/2}} - \frac{9}{8 n^{5/2}} +
  \frac{145}{128 n^{7/2}} + O(n^{-9/2}) \bigg).  \]
Division of the two expansions yields~\eqref{eq:asy-exp-r-branch}. In principle, any
number of terms would be available.

We also determine the variance by virtually the same approach. In this case we determine the
variance using the second factorial moment. Let $F_{r}^{(2)}(z)$ be the generating
function of the unnormalized second factorial moment of the number of $r$-branches, i.e.,
\[ F_{r}^{(2)}(z) = \sum_{n\geq 0} C_{n} \E \big( X_{n;r} (X_{n;r} - 1)\big) z^{n}.   \]
By analogous argumentation as before we know that $F_{0}^{(2)}(z)$ can be obtained by
differentiating the bivariate generating function $v B(vz)$ two times with respect to $v$
and setting $v = 1$. This gives
\begin{equation*}
\frac{\partial^2}{\partial v^2}vB(zv)\Big|_{v=1}=\frac{2z}{(1-4z)^{3/2}}=\frac{2u(1+u)}{(1-u)^3}.
\end{equation*}
Furthermore, we know that the recurrence
\[ F_{0}^{(2)}(z) = \frac{2z}{(1-4z)^{3/2}},\quad F_{r}^{(2)}(z) = \frac{z}{1-2z}
  F_{r-1}^{(2)}\Big(\frac{z^{2}}{(1 - 2z)^{2}}\Big)  \]
has to hold. Again, with the help of
Proposition~\ref{prop:recursion-solution}, we find
\[ F_{r}^{(2)}(z) = 2 \frac{1 - u^{2}}{u} \frac{u^{2^{r+1}}}{(1 - u^{2^{r}})^{4}},  \]
which can be locally expanded to 
\begin{multline*}
F_{r}^{(2)}(z)  = \frac1{2\cdot 16^r(1-4z)^{3/2}} - \frac {1+2\cdot {4}^{r}}{6\cdot
             16^r\sqrt{1-4z}} \\ -
             \frac{1 + 10\cdot 4^{r} - 11\cdot 16^{r}}{90\cdot 16^{r}} \sqrt{1 - 4z} +
             O((1 - 4z)^{3/2}).
\end{multline*}
After determining the asymptotic contribution of these coefficients by means of
singularity analysis and dividing the result by the asymptotic expansion of the Catalan
numbers, we arrive at an expansion for the second factorial moment:
\begin{equation*}
\E X_{n;r} (X_{n;r}-1) = \frac{1}{C_n} [z^n]F_{r}^{(2)}(z) = \frac{n^2}{16^r} + \frac{4-4^r}{3\cdot
  16^r} n + \frac{61 - 50\cdot 4^{r} - 11\cdot 16^{r}}{180\cdot 16^{r}} + O(n^{-1}).
\end{equation*}
By an elementary property of the second factorial moment, the variance can be computed by
means of $\V X_{n;r} = \E X_{n;r} (X_{n;r} - 1) + \E X_{n;r} - (\E X_{n;r})^{2}$. Doing so
yields~\eqref{eq:asy-var-r-branch} and thus, concludes the proof.
\end{proof}

Of course, the expected number of $r$-branches can also be computed explicitly by using
Cauchy's integral formula. This yields the following result:

\begin{proposition}
  \label{prop:explicit-exp-r-branch}
  The expected number $\E X_{n;r}$ of $r$-branches in binary trees of size $n$ is given by the explicit
  formula
  \begin{equation}
    \label{eq:explicit-exp-r-branch}
    \E X_{n;r} = \frac{n + 1}{\binom{2n}{n}} \sum_{\lambda \geq 1} \lambda
    \bigg[\binom{2n}{n + 1 - \lambda 2^{r}} - 2\binom{2n}{n - \lambda 2^{r}} +
    \binom{2n}{n - 1 - \lambda 2^{r}}\bigg].
  \end{equation}
\end{proposition}
\begin{proof}
  Note that all of the following integration contours are small closed curves that wind
  around the origin once; the substitution does not change that. Applying Cauchy's
  integral formula and using the substitution $z = \frac{u}{(1+u)^{2}}$ we obtain
  \begin{align*}
    [z^n]F_{r}^{(1)}(z)&=\frac1{2\pi i}\oint \frac{1-u^2}{u}\frac{u^{2^r}}{(1-u^{2^r})^2} \frac{dz}{z^{n+1}}\\
                   &=\frac1{2\pi i}\oint
                     \frac{1-u^2}{u}\frac{u^{2^r}}{(1-u^{2^r})^2}\frac{(1-u)(1+u)^{2n+2}}{(1+u)^3}\frac{du}{u^{n+1}}\\
                   &=\frac1{2\pi i}\oint \frac{(1-u)^2(1+u)^{2n}}{u^{n+2}}
                     \bigg(\sum_{\lambda\ge0}\lambda u^{\lambda 2^r}\bigg) du,
  \end{align*}
  where we used
  \begin{equation}
    \label{eq:expansion-1}
    \frac{x}{(1 - x)^{2}} = \sum_{\lambda \geq 0} \lambda x^{\lambda}.
  \end{equation}
  Then, interchanging summation and integration and applying Cauchy's
  integral formula once again yields
  \[ [z^n]F_{r}^{(1)}(z) = \sum_{\lambda\ge1}\lambda [u^{n+1-\lambda 2^r}](1-u)^2(1+u)^{2n}, \]
  which, after extracting the coefficient and dividing by $C_{n} = \frac{1}{n+1}
  \binom{2n}{n}$, proves the statement.
\end{proof}

We are also interested in the limiting distribution of $X_{n;r}$ for fixed
$r\in\N_{0}$ and $n\to\infty$. Note that as $X_{n;0} = n+1$ is a deterministic quantity, we
focus on the case that $r\geq 1$.

\begin{theorem}\label{thm:r-branches-limit}
  Let $r\in \N$ be fixed. Then $X_{n;r}$, the random variable modeling the number of
  $r$-branches in a binary tree of size $n$, is asymptotically normally distributed for
  $n\to\infty$. In particular, for $x\in\R$ we have
  \[ \P\Big(\frac{X_{n;r} - \E X_{n;r}}{\sqrt{\V X_{n;r}\,}} \leq x\Big) =
    \frac{1}{\sqrt{2\pi}} \int_{-\infty}^{x} e^{-t^{2}/2}~dt + O(n^{-1/2}). \]
\end{theorem}
\begin{remark}
  For the special case of $1$-branches, i.e., $r = 1$, a central limit theorem has been
  proved in~\cite{Wang-Waymire:1991:horton-ratios}. Additionally, numerical evidence for
  the validity of a general central limit theorem like the one we obtained above has been
  provided in \cite{Yamamoto-Yamazaki:2009:central-limit-bifurcation}.
\end{remark}
\begin{proof}[Proof of Theorem~\ref{thm:r-branches-limit}]
The central idea behind this proof is that $X_{n;r}$ can be interpreted as an
\emph{additive tree parameter}, meaning that the parameter can be evaluated as the sum of
the parameters corresponding to the subtrees rooted at the children of the root of the
original tree and an additional so-called \emph{toll function}.

In our case, it is straightforward to see that the number of $r$-branches in a binary tree
of size $n$ can be computed as the sum of the number of $r$-branches in the left and right
subtree. Only in the case where both subtrees have register function $r-1$, the root
itself is an $r$-branch that is not accounted for in the subtrees.

Hence, the random variable $X_{n;r}$ satisfies the distributional recurrence relation
\[ X_{n;r} = X_{I_{n};r} + X_{n-1-I_{n};r}^{*} + T_{n;r},  \]
where $X_{n;r}^{*}$ is an independent copy of $X_{n;r}$, $I_{n;r}$ is a random variable
modeling the size of the left subtree with
\[ \P(I_{n} = j) = \frac{C_{j}C_{n-1-j}}{C_{n}}\qquad \text{where } j\in\{0,1,\ldots,
  n-1\},  \]
and where $T_{n;r}$ is a toll function depending on $X_{n;r}$ satisfying
\[ T_{n;r} = \begin{cases} 1 & \text{if the register function of both rooted subtrees is
    } r-1,\\ 0 & \text{otherwise.}\end{cases}  \]
Asymptotic normality of $X_{n;r}$ can now be obtained by showing that the expectation of
the toll function decays exponentially, according to \cite{Wagner:2015:centr-limit}.

In order to show that this condition is satisfied we consider $B_{r}^{=}(z)$, the
generating function for binary trees with register function equal to
$r$. By~\eqref{eq:trees:equal-ogf} we have
\[ B_{r}^{=}(z) = \frac{1 - u^{2}}{u} \frac{u^{2^{r}}}{1 - u^{2^{r+1}}}.  \]

By means of Property~(\ref{prop:subs-prop:pole-translation}) in
Proposition~\ref{prop:subs-prop} we can write
\[ B_{r}^{=}(z) = \frac{(u-1)(u+1) u^{2^{r} - 1}}{\prod_{0\leq k < 2^{r+1}} (u -
    \omega_{k})} = -\frac{z^{2^{r} - 1} \prod_{0<k<2^{r}} Z(\omega_{k})}{\prod_{0<k<2^{r}}
  (z - Z(\omega_{k}))}, \]
where $\omega_{k} := \exp(2\pi i k / 2^{r+1})$. In particular, this proves that
$B_{r}^{=}(z)$ is a rational function. As we have $Z(\omega_{k}) = \frac{1}{2 +
  2\cos(\pi k/2^{r})}$, the dominant singularity of $B_{r}^{=}(z)$ can uniquely be
identified as $Z(\omega_{1}) = \frac{1}{2 + 2\cos(\pi/2^{r})} > 1/4$, which proves that
the ratio of trees with register function $r$ among all binary trees of size $n$ decays
exponentially.

The exponential decay of the expected value of the toll function $T_{n;r}$ now follows
from the fact that $\E T_{n;r}$ equals the ratio of the trees whose children both have
register function $r-1$ among all trees of size $n$. These trees form a subset of those
counted by $B_{r}^{=}(z)$, which means that their ratio has to decay exponentially as
well.

The asymptotic normality of $X_{n;r}$ now follows from \cite[Theorem
2.1]{Wagner:2015:centr-limit}. All that remains to show is that the speed of convergence
is $O(n^{-1/2})$. In order to do so, we observe that the proof for asymptotic normality in
Wagner's theorem basically relies on \cite[Theorem 2.23]{Drmota:2009:random}, which uses a
formulation of Hwang's Quasi-Power Theorem without quantification of the speed of
convergence (cf.~\cite[Theorem 2.22]{Drmota:2009:random}). By replacing this argument with
a quantified version (cf.~\cite{Hwang:1998} or \cite{Heuberger-Kropf:2016:higher-dimen}
for a generalization to higher dimensions) of the Quasi-Power Theorem, we find that the
speed of convergence in Wagner's result---and therefore, in our result as well---is
$O(n^{-1/2})$.
\end{proof}

\subsection{The total number of branches}
\label{sec:all-branches}

So far, we were dealing with fixed $r$, and the number of $r$-branches in trees of size
$n$, for large $n$. Now we consider the total number of such branches, i.e., the sum over
$r\ge0$, which was not considered in \cite{YaYa10}. Formally, this corresponds to the
analysis of the random variable $X_{n}\colon \mathcal{B}_{n}\to \N_{0}$ where
\[ X_{n} := \sum_{r\geq 0} X_{n;r}.  \]
By definition, $X_{n}$ enumerates the total number of branches in binary trees of size $n$.

With the help of the bounds for the number of $r$-branches obtained in
Proposition~\ref{prop:r-branches-range} we can characterize the range of $X_{n}$ as well.

\begin{proposition}\label{prop:total-branches-bounds}
  Let $n\in \N_{0}$ and let $w_{2}(n)$ denote the binary weight, i.e.\ the number of
  non-zero digits in the binary expansion of $n$. Then the bound 
  \[ n + 1 + \llbracket n > 0 \rrbracket \leq X_{n} \leq 2n + 2 - w_{2}(n+1) \leq 2n+1 \]
  holds for the random variable $X_{n}$ and is sharp.
\end{proposition}
\begin{remark}
  The sharp upper bound $2n+2 - w_{2}(n+1)$ is enumerated by sequence
  \href{http://oeis.org/A005187}{A005187}, shifted by one, in \cite{OEIS:2016}.
\end{remark}
\begin{proof}
  We begin by observing that for fixed $n\in \N$, the random variable $X_{n;r}$ vanishes
  for sufficiently large $r$. As the bounds from Proposition~\ref{prop:r-branches-range} are
  sharp, we are allowed to sum up the inequalities in order to obtain
  \[ n+1 + \llbracket n > 0 \rrbracket \leq \sum_{r\geq 0} X_{n;r} \leq \sum_{r\geq 0}
    \Big\lfloor \frac{n+1}{2^{r}} \Big\rfloor.  \]
  This immediately proves the lower bound from the statement.

  In order to prove the upper bound we investigate the sum $\sum_{r\geq 0} \lfloor m/2^{r}
  \rfloor$ for $m\in \N$. Consider the binary digit expansion of $m$, denoted by
  $(x_{k}\dots x_{1}x_{0})_{2}$. In this context, the sum can be written as
  \begin{align*}
    \sum_{r=0}^{k} (x_{k}\dots x_{r+1}x_{r})_{2} & = \sum_{r=0}^{k}
    x_{r} (1 + 2 + 4 + \dots + 2^{r}) = \sum_{r=0}^{k} x_{r} (2^{r+1} - 1) \\
    & = 2\cdot (x_{k}\dots x_{1}x_{0})_{2} - \sum_{r=0}^{k} x_{r} = 2m - w_{2}(m).
  \end{align*}
  By setting $m = n+1$ we see that the upper bound holds as well. The fact that $2n+2 -
  w_{2}(n+1) \leq 2n+1$ is a direct consequence of $w_{2}(n+1) \geq 1$ for all
  $n\in\N_{0}$.

  It is easy to see that the bounds are sharp for $n = 0$. For $n > 0$, the lower bound is
  attained by any chain of size $n$: they consist of $n+1$ leaves (which are $0$-branches)
  and exactly one additional $1$-branch which connects all the leaves. The upper bound is
  attained by the family of almost complete binary trees constructed in the proof of
  Proposition~\ref{prop:r-branches-range}, which follows from the fact that $X_{n;r}$ attains its
  maximum in the tree $B_{n+1}$, which does not depend on the value of $r$.
\end{proof}

First, to get an explicit formula, the
results from Proposition~\ref{prop:explicit-exp-r-branch} can be summed.

\begin{corollary}\label{cor:explicit-exp-branches} 
  The expected number of branches in binary trees of size $n$ is
  given by the explicit formula
  \begin{equation*}
    \E X_{n} = \frac{n+1}{\binom{2n}{n}}\sum_{k=1}^{n+1} (2 - 2^{-v_{2}(k)}) k \bigg[\binom{2n}{n+1-k}-2\binom{2n}{n-k}+
    \binom{2n}{n-1-k}\bigg],
  \end{equation*}
  where $v_{2}(k)$ is the dyadic valuation of $k$, i.e., the highest exponent
  $\nu$ such that $2^\nu$ divides $k$.
\end{corollary}
\begin{proof}
To simplify the double summation, we consider
\begin{align*}
\psi(k):=\sum_{\substack{\lambda\ge0,\, r\ge0:\\ \lambda 2^r=k}}\lambda.
\end{align*}
This sum can be simplified to some degree. We write $k = 2^{v_{2}(k)} (2j + 1)$, such that
we have
\[ \psi(k) = \sum_{r=0}^{v_{2}(k)} 2^{v_{2}(k) - r} (2j+1) = (2^{v_{2}(k) + 1} -
  1)(2j+1) = (2 - 2^{-v_{2}(k)}) k,  \]
which proves the result.
\end{proof}

While it is absolutely possible to work out the asymptotic growth from this explicit
formula, at it was done in earlier papers \cite{Flajolet-Raoult-Vuillemin:1979:register, Kemp79}, we choose a faster
method, like in \cite{Flajolet-Prodinger:1986:regis}. It works on the level of generating functions and uses the
Mellin transform together with singularity analysis of generating functions~\cite{Flajolet-Sedgewick:ta:analy,
  bona:prodinger:2015:analyt}.

The following theorem describes the asymptotic behavior for the expected number of
branches in a binary tree.

\begin{theorem}
  \label{thm:asy-branches}
  The expected value of the total number of branches in a random binary tree of size $n$
  admits the asymptotic expansion
  \begin{equation*}
    \E X_{n} = \frac{4n}{3} + \frac{1}{6}\log_{4} n - \frac{2 \zeta'(-1)}{\log2} -
    \frac{\gamma}{12\log 2} - \frac{1}{6\log 2} + \frac{43}{36}
    +\delta(\log_4n)+ O \Big(\frac{\log n}{n}\Big),
  \end{equation*}
  where 
  \begin{equation*}
    \delta(x) := \frac1{\log2}\sum_{k\neq 0}
    \Gamma\Big(\frac{\chi_k}{2}\Big)\zeta(\chi_k-1)(\chi_k-1)e^{2\pi i k x}
  \end{equation*}
  is a $1$-periodic function of mean zero, given by its Fourier series expansion with
  $\chi_{k} = \frac{2\pi i k}{\log 2}$.
\end{theorem}
\begin{remark}
  Note that the value of the derivative of the zeta function is given by $\zeta'(-1) = -
  \frac{1}{12} - \log A \approx -0.1654211437$, where $A$ is the Glaisher-Kinkelin
  constant (cf.~\cite[Section 2.15]{Finch:constants:2003}).
\end{remark}
\begin{figure}[ht]
  \centering
  \begin{tikzpicture}
    \begin{axis}[width=14cm, height=8cm,
      yticklabel style={/pgf/number format/fixed, /pgf/number format/precision=3},
      xtick = {2,2.5,...,5}, legend entries = {empirical, Fourier series},
      ytick = {-0.09, -0.06, ..., 0.06}, legend pos = south east]
      \addplot+[only marks, mark size=0.4pt] table[x=x, y=empirical] {theorem2_fluc.dat};
      \addplot+[no marks] table[x=x, y=fourier] {theorem2_fluc.dat};
    \end{axis}
  \end{tikzpicture}
  \caption{Partial Fourier series (20 summands) compared with the empirical values of the
    function $\delta$ from Theorem~\ref{thm:asy-branches}}
  \label{fig:branches-fluc}
\end{figure}
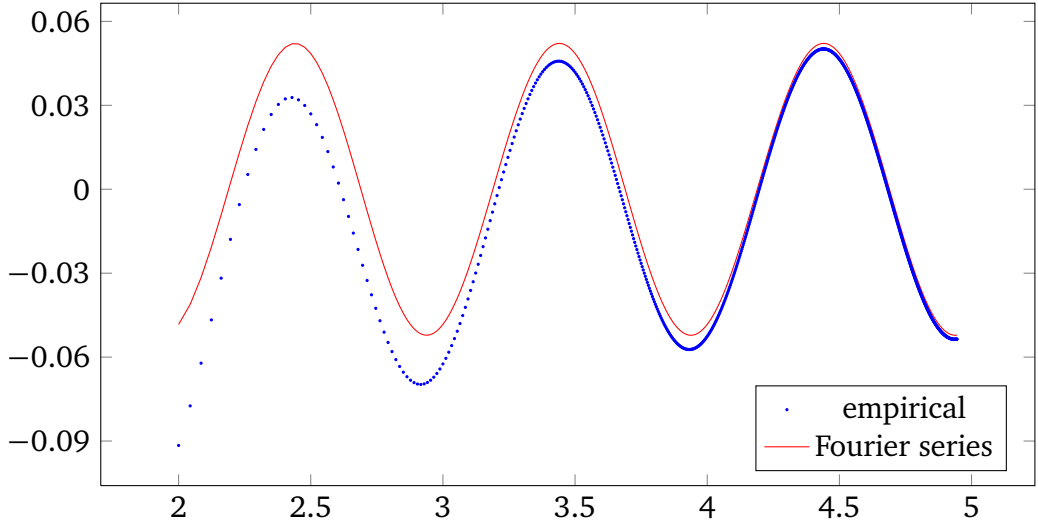
\begin{remark}
  The occurrence of the periodic fluctuation $\delta$ where the argument is logarithmic in
  $n$ is actually not surprising: while this phenomenon is already very common in the
  context of the register function, fluctuations appear very often in the asymptotic
  analysis of sums. 
\end{remark}
\begin{proof}
By using~\eqref{eq:trees:r-branch-expectation-gf} and~\eqref{eq:expansion-1}, the generating function of interest can be written
as
\begin{align*}
F(z) = \sum_{r\geq 0} F_{r}^{(1)}(z) = \sum_{r\ge0}\frac{1-u^2}{u}\frac{u^{2^r}}{(1-u^{2^r})^2}
=\frac{1-u^2}{u}\sum_{r,\lambda\ge0}\lambda u^{\lambda 2^r}.
\end{align*}
To find the asymptotic behavior of the sum, we set $u=e^{-t}$, consider the function
\begin{equation*}
f(t):=\sum_{r,\lambda\ge0}\lambda e^{-t\lambda 2^r},
\end{equation*}
and compute its Mellin transform
\begin{equation*}
f^*(s) = \sum_{r,\lambda\ge0} \lambda^{1-s} 2^{-rs} \Gamma(s)=
\Gamma(s)\zeta(s-1)\frac1{1-2^{-s}}.
\end{equation*}
The fundamental strip of $f^{*}(s)$ is $\langle 2, \infty\rangle$. Then, by the inversion formula for the Mellin transform we obtain
\begin{equation}\label{eq:mellin-inverse-integral}
f(t)=\frac1{2\pi i}\int_{5-i\infty}^{5+i\infty}\Gamma(s)\zeta(s-1)\frac1{1-2^{-s}}t^{-s}~ds,
\end{equation}
which is valid for real, positive $t \to 0$, which gives an expansion for $u \to 1^{-}$,
or, equivalently $z\to (1/4)^{-}$. In order to use this representation of $f(t)$ for our
purposes (i.e.\ in order to apply singularity analysis), we need to have analyticity in a
larger region (cf.~\cite{Flajolet-Odlyzko:1990:singul}), e.g.\ in a complex punctured neighborhood of $1/4$
with\footnote{Note that the bound $2\pi /5$ is somewhat arbitrary: the argument just needs
  to be less than $\pi/2$.} $\abs{\arg(z - 1/4)} > 2\pi/5$.
In particular, the expansion
\begin{equation}\label{eq:t-to-z}
t = -\log(U(z)) =
2\sqrt{1-4z}+\frac23(1-4z)^{3/2} + O((1 - 4z)^{5/2})
\end{equation}
implies
\[ \abs{\arg t} = \frac{1}{2} \abs{\arg(1 - 4z)} + o(1), \]
such that we have the bound $\abs{\arg t} < 2\pi /5$ for $t\to 0$, given that the
restriction on the argument in the $z$-world is satisfied.

Then, given that $\Re(s) = 5$ or $\Re(s) = -3$ holds we find that we have the estimate
\begin{equation}\label{eq:growth-estimate}
\abs{f^{*}(s) t^{-s}} = O\Big(\abs{\Im(s)}^{5} \abs{t}^{-\Re(s)}
  \exp\Big(-\frac{\pi}{10} \abs{\Im(s)}\Big)\Big)
\end{equation}
for the integrand in~\eqref{eq:mellin-inverse-integral}. The very same estimate also holds
for $-3 \leq \Re(s) \leq 5$ where $\Im(s) = \frac{2\pi i}{\log 2} \big(k +
\frac{1}{2}\big)$ and $k\in \Z$ tends towards $\infty$ or $-\infty$. This is a consequence
of the behavior of $\Gamma(s)$ as given in \DLMF{5.11}{3}, estimates for $\zeta(s)$ as given in
\cite[13.51]{Whittaker-Watson:1996}, and the fact that $\frac{1}{1 - 2^{-s}}$ is bounded
for $s$ in the given ranges.

Together with the identity theorem for analytic functions (cf.~\cite{Flajolet-Prodinger:1986:regis} for a similar
argumentation) this means that the inverse Mellin transform remains valid for complex $z$
in a neighborhood of $1/4$ with $\abs{\arg(1 - 4z)} > 2\pi/5$, which justifies the
following approach.

We can evaluate~\eqref{eq:mellin-inverse-integral} by shifting the line of integration
from $\Re(s) = 5$ to $\Re(s) = -3$ and collecting the residues of the poles we cross. This
yields
\[ f(t) =  \sum_{p\in P} \Res_{s=p}(f^{*}(s)t^{-s}) + \frac{1}{2\pi i}
  \int_{-3-i\infty}^{-3+i\infty} f^{*}(s)t^{-s}~ds, \]
where $P = \{-2, 0, 2\} \cup \{\chi_{k} \mid k\in \Z\setminus\{0\}\}$ and $\chi_{k} :=
\frac{2\pi i k}{\log 2}$. Multiplying this representation of $f(t)$ with $\frac{1 -
  u^{2}}{u}$ and expanding everything locally for $z\to 1/4$ yields a singular expansion
from which coefficient growth can be extracted by means of singularity analysis.

For the error term we use the estimate above and find
\[ \frac{1}{2\pi i} \int_{-3-i\infty}^{-3+i\infty} f^{*}(s) t^{-s}~ds =
   O(|t|^{3}). \]
However, for the sake of simplicity we want to use the contribution of the residue
collected from the pole at $s = -2$ as the error term. We immediately find
\[ \Res_{s = -2}(f^{*}(s) t^{-s}) = O(\abs{t}^{2}).  \]

We compute the remaining singularities explicitly with the help of
SageMath~\cite{SageMath:2016:7.4} and obtain
\begin{align*}
  f(t) & = \sum_{p\in P\setminus\{-2\}} \Res_{s=p}(f^{*}(s) t^{-s}) + O(\abs{t}^{2}) \\ &=
  \Big(\frac{4}{3t^{2}} + \frac{\log t}{12\log 2}
  + \frac{\zeta'(-1)}{\log 2} + \frac{\gamma}{12\log 2} - \frac{1}{24}\Big) +
  \sum_{k\neq 0} \frac{\Gamma(\chi_{k}) \zeta(\chi_{k} - 1)}{\log 2} t^{-\chi_{k}} + O(\abs{t}^{2}),
\end{align*}
where we used the Laurent expansion for the gamma function at $s = 0$
(cf.~\cite[43:6:1]{Oldham-Spanier:1987:atlas-funct}) and the fact that $\zeta(-1) = -1/12$
(cf.~\DLMF{25.6}{3}). When translating this expansion in terms of $t \to 0$ to an
expansion in terms of $z\to 1/4$, we have to be particularly careful with respect to the
sum of the residues at $s = \chi_{k}$ as we have to check that the sum of the errors is
still controllable.

We do so by considering the expansion
\[ t^{-\chi_{k}} = (1-4z)^{-\chi_{k}/2} \big(1 + O(1-4z)\big)^{-\chi_{k}/2}.  \]
With the well-known inequality
\[ \abs{\exp(z) - 1} \leq \abs{z} \exp{\abs{z}}  \]
we find
\begin{align}\label{eq:error-sum-estimate}
  \abs{(1+O(1-4z))^{-\chi_{k}/2} - 1} &= \abs[\Big]{\exp\Big(-\frac{\chi_{k}}{2}
                                        \log(1+O(1-4z))\Big) - 1} \notag\\
  & \leq
  \abs[\Big]{\frac{\chi_{k}}{2}} \abs{\log(1+O(1-4z))} \exp\Big(\frac{2\pi}{\log 2} \abs{k}
  \abs{\log(1+O(1-4z))}\Big)\notag\\
  & = \abs[\Big]{\frac{\chi_{k}}{2}} O(1-4z) \exp\Big(\frac{2\pi}{\log 2}
  \abs{k} O(1-4z)\Big).
\end{align}
This proves that the errors we sum up are of order $O(\abs{k} (1-4z) \exp(\abs{k}
O(1-4z)))$. Thus, if $z$ is chosen sufficiently close to $1/4$, this exponential growth is
slow enough to vanish within the exponential decay established
in~\eqref{eq:growth-estimate}.

Finally, by considering
\[ \frac{1 - u^{2}}{u} = 4 \sqrt{1 - 4z} + 4 (1 - 4z)^{3/2} + O((1-4z)^{5/2})  \]
we find that
\begin{multline*} 
  F(z) = \frac{4}{3 \sqrt{1 - 4z}} + \bigg(\frac{\gamma}{3 \log 2} + \frac{4 \zeta'(-1)}{\log 2}
  + \frac{11}{18} + \frac{\log(1-4z)}{6 \log 2}\bigg) \sqrt{1 - 4z} \\ + \frac{4}{\log 2} \sum_{k\neq 0}
  \Gamma(\chi_{k}) \zeta(\chi_{k} - 1) (1-4z)^{1/2 - \chi_{k}/2} + O((1 - 4z)^{3/2} \log(1-4z)).
\end{multline*}
Applying singularity analysis, normalizing the result by $C_{n}$ and rewriting the
coefficients of the contributions from the poles at $\chi_{k}$ via the duplication formula
for the Gamma function (cf.~\DLMF{5.5}{5}) then proves the asymptotic expansion for $\E X_{n}$.
\end{proof}

While this multi-layer approach enabled us to analyze the expected value of the number of
branches in binary trees of size $n$, the same strategy fails for computing the
variance. This is because the random variables modeling the number of $r$-branches are
correlated for different values of $r$---and thus, the sum of the variances (which we
compute by our approach) differs from the variance of the sum.

This concludes our study of the number of branches per binary tree. In the next section,
we analyze a quantity that has similar properties as the register function, but is defined on
simple two-dimensional lattice paths.

\section{Reduction of Lattice Paths}
\label{sec:paths}

\subsection{Iterative Reductions and an Analogue to the Register Function}

Recall that the register function describes the number of reductions of a binary
tree required in order to reduce the tree to a leaf. By defining a similar
process for simple two-dimensional lattice paths, a
function that plays a similar role as the register function is obtained.

Simple two-dimensional lattice paths are sequences of the symbols $\{\upa,
\righta, \downa, \lefta\}$. It is easy to see that the generating function
counting these paths (without the path of length $0$) is
\[ L(z) = \frac{4z}{1 - 4z} = 4z + 16z^{2} + 64z^{3} + 256z^{4} + 1024z^{5} + \cdots.  \]

\begin{proposition}\label{prop:lat-relation}
  The generating function $L(z) = \frac{4z}{1-4z}$ satisfies the functional equation
  \begin{equation}\label{eq:lat-relation}
    L(z) = 4 L\Big(\frac{z^{2}}{(1 - 2z)^{2}}\Big) + 4z.
  \end{equation}
\end{proposition}
\begin{remark}
  It is easy to verify this result by means of substitution and expansion. However, we want
  to give a combinatorial proof---this approach also motivates the definition of a
  recursive generation process for lattice paths, similar to the process for binary trees
  from above.
\end{remark}
\begin{figure}[ht]
    \centering
    \includegraphics[scale=1]{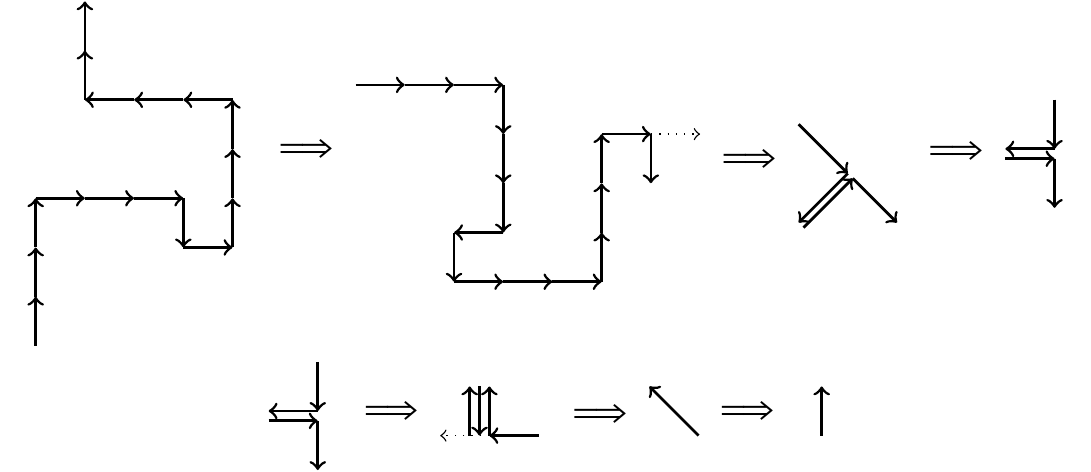}
    \caption{Repeated application of the reduction $\Phi_{L}$ on a path with reduction degree 2}
    \label{fig:lat-red-vis}
  \end{figure}
\begin{proof}
  We show that the right-hand side of \eqref{eq:lat-relation} counts simple
  two-dimensional lattice paths (excluding the path of length $0$) as well. In order to do
  so, we introduce a reduction of lattice paths denoted by $\Phi_{L}$, that works on a
  given path $\ell$ with length $\geq 2$ as follows:

  First, if $\ell$ starts vertically (i.e.\ with $\upa$ or $\downa$), rotate the
  entire path clockwise such that it starts horizontally.

  Second, if the (possibly rotated) path ends horizontally, rotate the very last step
  (which has to be $\righta$ or $\lefta$) once again clockwise.

  Now, the path can be reduced by collapsing each pair of horizontal-vertical path
  segments into a path of length 1 as follows:
  \begin{itemize}
  \item If a segment starts with $\righta$ and the first vertical step is $\upa$, replace it by
    $\nearrow$,
  \item if a segment starts with $\righta$ and the first vertical step is $\downa$, replace it
    by $\searrow$,
  \item if a segment starts with $\lefta$ and the first vertical step is $\downa$, replace it by
    $\swarrow$,
  \item and if a segment starts with $\lefta$ and the first vertical step is $\upa$, replace it
    by $\nwarrow$.
  \end{itemize}
  Finally, rotate the obtained path with the diagonal steps by $45^{\circ}$ clockwise such that
  $\nearrow$ becomes $\righta$ and so on. The resulting path is the reduction
  $\Phi_{L}(\ell)$. This process is visualized in Figure~\ref{fig:lat-red-vis}.

  As it is the case with the reduction $\Phi$ of binary trees, $\Phi_{L}$ is a partial
  function as well: $\Phi_{L}(s)$ is undefined for $s\in \{\upa, \righta,
  \downa, \lefta\}$. Furthermore, a reduced path can be expanded (although not
  uniquely) to its original path again by rotating a given path to the left such that it
  is given in diagonal steps, reading the replacements from above from right to left, and
  then optionally rotating the very last step and/or the entire path to the left again.

  We find that the generating function for lattice paths consisting of sequences of
  horizontal-vertical segments is given by $L\big(\frac{z^{2}}{(1 - 2z)^{2}}\big)$. In
  order to see this, consider the expansion of a path of length 1, for example
  \[ \righta \quad\Longrightarrow\quad \nearrow \quad\Longrightarrow\quad \righta (\righta + \lefta)^{*} \upa (\upa +
    \downa)^{*}, \]
  where the regular expression describes a sequence of horizontal steps starting with
  $\righta$, followed by a sequence of vertical steps starting with $\upa$. As path length
  is marked by $z$, the expansion above translates to the substitution
  \[ z \mapsto \frac{z^{2}}{(1-2z)^{2}}. \]
  
  As all four expansion variants lead to the same variable substitution,
  $L\big(\frac{z^{2}}{(1 - 2z)^{2}}\big)$ precisely enumerates all lattice paths
  consisting of sequences of horizontal-vertical segments.

  The factor $4$
  in~\eqref{eq:lat-relation} is explained by the four path variants obtained by either
  rotating just the last step and/or the entire path.

  Putting all of this together, \eqref{eq:lat-relation} can be interpreted combinatorially
  as the following statement: a simple two-dimensional lattice path is either a simple
  step, or can be obtained by expanding another simple two-dimensional lattice path. This
  proves the proposition.
\end{proof}

The process described in the proof of Proposition~\ref{prop:lat-relation} allows us to
assign a unique number to each lattice path:
\begin{definition}
  Let $\ell$ be a simple two-dimensional lattice path consisting of at least one step. We
  define the \emph{reduction degree} of $\ell$, denoted as $\rdeg(\ell)$ as
  \[ \rdeg(\ell) = n \quad \iff\quad \Phi_{L}^{n}(\ell) \in \{\upa, \righta,
    \downa, \lefta\}.  \]
\end{definition}
\begin{remark}
  The parallels between the reduction degree and the register function are obvious:
  both count the number of times some given mathematical object can be reduced according
  to some rules until an atomic form of the respective object is obtained. Therefore, both
  functions describe, in some sense, the complexity of a given structure.
\end{remark}

In the remainder of this section we want to derive some asymptotic results for the
reduction degree, namely the expected degree of a lattice path of given length as
well as the corresponding variance.

Analogously to our strategy for~\eqref{eq:u1}, we want to interpret~\eqref{eq:lat-relation} as
a recursive generation process as well and therefore set 
\begin{equation*}
  \label{eq:lat-recursion}
  L_{0}(z) = 4z,\quad L_{r}(z) = 4 L_{r-1}\Big(\frac{z^{2}}{(1 - 2z)^{2}}\Big) + 4z,\quad
  r\geq 1.
\end{equation*}
This yields the functions
\begin{align*}
  L_{1}(z) & = 4 z + 16 z^{2} + 64 z^{3} + 192 z^{4} + 512 z^{5} + 1280 z^{6} + 3072 z^{7} + 7168 z^{8}
  + \cdots, \\
  L_{2}(z) & = 4 z + 16 z^{2} + 64 z^{3} + 256 z^{4} + 1024 z^{5} + 4096 z^{6} + 16384 z^{7} + 65280
  z^{8} + \cdots, \\
  L_{3}(z) & = 4 z + 16 z^{2} + 64 z^{3} + 256 z^{4} + 1024 z^{5} + 4096 z^{6} + 16384 z^{7} + 65536
  z^{8} + \cdots.
\end{align*}
Due to the construction, the function $L_{r}(z)$ is the generating functions of those
lattice paths with reduction degree $\leq r$.

By using Proposition~\ref{prop:recursion-solution} with $D(z) = 4z$ and $E(z) = 4$, the
generating functions $L_{r}(z)$ can be written explicitly in terms of $u$ as
\[ L_{r}(z) = \sum_{j=0}^{r} 4^{j+1}
  \frac{u^{2^{j}}}{(1 + u^{2^{j}})^{2}}.  \]
The generating function $L_{r}^{=}(z)$ of lattice paths with reduction degree equal
to $r$ can then be found by considering the difference $L_{r}(z) - L_{r-1}(z)$, or,
alternatively, by dropping the summand $4z$ in the recursion above. Both approaches lead to
\begin{equation}
  \label{eq:lat-gen-r}
  L_{r}^{=}(z) = 4^{r+1} \frac{u^{2^{r}}}{(1 + u^{2^{r}})^{2}}.
\end{equation}

The coefficients of this function can be extracted explicitly by applying Cauchy's
integral formula.
\begin{proposition}\label{prop:lat-coef-r}
  The number of two-dimensional simple lattice paths of length $n$ that have
  reduction degree $r$ is given by
  \[ [z^{n}] L_{r}^{=}(z) = 4^{r+1} \sum_{\lambda \geq 0} \lambda (-1)^{\lambda - 1}
    \bigg[\binom{2n-1}{n - \lambda 2^{r}} - \binom{2n-1}{n - \lambda 2^{r} - 1}\bigg].  \]
\end{proposition}
\begin{proof}
  The proof is straightforward and uses the same approach as the proof of
  Proposition~\ref{prop:explicit-exp-r-branch}.
\end{proof}

In fact, by studying the substitution $z = Z(u)$ closely, the asymptotic behavior of
the coefficients of $L_{r}^{=}(z)$ can be extracted as well.

\begin{proposition}\label{prop:L-partial-fraction}
  Let $r\geq 1$ be fixed. Then $L_{r}^{=}(z)$ is a rational function in $z$ with poles at
  \[ z_{k} = \frac{1}{4\cos^{2}(\pi k 2^{-r-1})}  \]
  with singular expansions
  \[ L_{r}^{=}(z) = \frac{4 \tan^{2}\big(\frac{k\pi}{2^{r+1}}\big)}{\Big(1 -
      \frac{z}{z_{k}}\Big)^{2}} - \frac{4 \sin^{2}\big(\frac{k\pi}{2^{r+1}}\big) +
      2}{\cos^{2}\big(\frac{k\pi}{2^{r+1}}\big)} \frac{1}{1 - \frac{z}{z_{k}}} +
    O(1),\qquad z\to z_{k},  \]
  for $1\leq k < 2^{r}$, $k$ odd.
\end{proposition}
\begin{proof}
  First note that all of the following estimates are not uniform w.r.t.\ $r$, meaning that the
  constant in the $O$-term depends heavily on $r$.

  From \eqref{eq:lat-gen-r} and Proposition~\ref{prop:subs-prop}(\ref{prop:subs-prop:pole-translation}), it is
  immediately clear that $L_{r}^{=}(z)$ is a rational function
  in $z$ with poles at $Z(\omega)$ where $\omega$ runs through the $2^{r}$th roots of
  $-1$. By symmetry, we restrict ourselves to $\omega$ with $\Im\omega \leq 0$.

  We now fix such an $\omega = \exp(-k\pi i 2^{-r})$ for some $1\leq k < 2^{r}$, $k$
  odd. By expansion around $\omega$, we get
  \[ \frac{4^{r+1} u^{2^{r}}}{(1+u^{2^{r}})^{2}} = -\frac{4\omega^{2}}{(u - \omega)^{2}}
    - \frac{4\omega}{u - \omega} + O(1)\quad \text{ for } u\to \omega.  \]
  We know that $L_{r}^{=}(z)$ has a pole of order $2$ at $z_{k} = Z(\omega)$, implying
  that expanding $L_{r}^{=}(z)$ for $z\to z_{k}$ yields an expansion of the form
  \[ L_{r}^{=}(z) = \frac{A}{\Big(1- \frac{z}{z_{k}}\Big)^{2}} + \frac{B}{1 -
      \frac{z}{z_{k}}} + O(1) \quad \text{ for } z\to z_{k}  \]
  where $A$ and $B$ are some constants depending on $k$ and $r$.
  With the help of Cauchy's integral formula, the substitution $u = U(z)$, and the
  expansion from above we can determine the constants $A$ and $B$ and find
  \[ L_{r}^{=}(z) = \frac{-4 (\omega -1)^{2}}{(\omega + 1)^{2}} \frac{1}{\Big(1 -
      \frac{z}{z_{k}}\Big)^{2}} + \frac{4 (\omega^{2} - 4\omega + 1)}{(\omega + 1)^{2}}
    \frac{1}{1 - \frac{z}{z_{k}}} + O(1) \quad \text{ for } z\to z_{k}.  \]

  Rewriting all complex exponentials in terms of trigonometric functions then yields the
  result.
\end{proof}

With the help of this characterization of the poles of $L_{r}^{=}$ the asymptotic behavior
of the number of lattice paths with reduction degree equal to $r$ can be obtained.

\begin{corollary}\label{cor:L-asy-expansion}
  Let $r\geq 1$ be fixed. The number of lattice paths with reduction degree equal to $r$
  admits the asymptotic expansion
  \begin{multline}
    \label{eq:lat-comp-equal}
    [z^{n}]L_{r}^{=}(z) = (4\cos^{2}(\pi 2^{-r-1}))^{n} \Big(4\tan^{2}(\pi 2^{-r-1})n -
    \frac{2}{\cos^{2}(\pi 2^{-r-1})}\Big) \\ + O\Big( (4\cos^{2}(3\pi 2^{-r-1}))^{n} n\Big),
  \end{multline}
  where the constant in the $O$-term depends on $r$.
\end{corollary}
\begin{proof}
  We use the notation of Proposition~\ref{prop:L-partial-fraction}. By means of singularity
  analysis and by considering that $L_{r}^{=}$ is a rational function, we find that the pole
  at $z_{k}$ (for odd $k$) yields a contribution of (up to simplification)
  \begin{multline*} z_{k}^{-n}\Big(4\tan^{2}(k\pi 2^{-r-1})(n+1) - \frac{4\sin^{2}(k\pi 2^{-r-1}) +
      2}{\cos^{2}(k\pi 2^{-r-1})}\Big) \\= z_{k}^{-n} \Big(4\tan^{2}(k\pi 2^{-r-1})n -
    \frac{2}{\cos^{2}(k\pi 2^{-r-1})}\Big)
  \end{multline*}
  for sufficiently large $n$.
\end{proof}

We turn to the investigation of the expected
reduction degree. Let $\mathcal{L}_{n}$ denote the set of simple two-dimensional lattice
paths of size $n$. Consider the family of random variables $D_{n}\colon \mathcal{L}_{n} \to
\mathbb{N}_{0}$ modeling the reduction degree of the lattice paths of length $n$
under the assumption that all paths are equally likely.

Similar to the investigations we have conducted for the random variables in
Sections~\ref{sec:r-branches} and~\ref{sec:all-branches}, we want to characterize the
range of the reduction degree for lattice paths of given length $n$ as well.

\begin{proposition}\label{prop:paths:rdeg-range}
  Let $n \in \N$. Then the reduction degree for any simple two-dimensional lattice path of
  length $n$ satisfies
  \[ \llbracket n > 1 \rrbracket \leq D_{n} \leq \lfloor \log_{2}n\rfloor,  \]
  and these bounds are sharp.
\end{proposition}
\begin{proof}
  First, observe that for $n = 1$, we only have the atomic steps $\{\upa, \righta, \downa,
  \lefta\}$, and all of them have reduction degree $0$, which the lower and upper bound
  given above agree upon.

  For $n > 1$ we find that, e.g., the path $(\righta)^{n}$ has reduction degree $1$. In
  combination with the fact that there are no paths of length greater than $1$ with
  reduction degree $0$, this establishes the lower bound and proves that it is sharp.

  In order to prove the upper bound, we consider $M$ to be the maximal reduction degree
  among all lattice paths of length $n$, i.e.\ the corresponding path can be obtained from
  one of the steps (of length $1$) by expanding the path $M$ times.

  The shortest possible path after $M$ expansions can be obtained by replacing every step
  of the path iteratively by a segment of length $2$, meaning that the length doubles after
  every expansion. Thus, a minimally expanded path has length $2^{M}$.

  As the minimally expanded path has to be at most equally long as the original path, the
  inequality $2^{M} \leq n$ and therefore $M \leq \lfloor \log_{2} n \rfloor$ holds, which
  proves the upper bound.

  In order to construct a path of length $n$ with reduction degree equal to $\lfloor
  \log_{2} n\rfloor$, we consider the binary digit expansion $(x_{k}\ldots
  x_{1}x_{0})_{2}$ of $n$. Reading this expansion from left to right, starting at $x_{k
    - 1}$, we construct the path as follows: we start with $\righta$, if the current
  digit is $0$ then we expand the path minimally, and otherwise we expand all but the last
  step of the path minimally; the last step is expanded by replacing it by a corresponding
  segment of length $3$ (i.e.\ one additional step is added in contrast to minimal
  expansion). The digit $x_{k} = 1$ is not relevant for this construction, thus it is ignored.

  It is easy to see that the length of the resulting path is $n$, as our construction
  corresponds to the ``double-and-add''-strategy used to determine the value of the binary
  expansion. Furthermore, for each of the digits in $x_{k - 1}\ldots x_{1}x_{0}$ we
  have expanded our path once, which produces a path with reduction degree
  $k$. Finally, from the binary expansion it is easy to see that $k = \lfloor
  \log_{2} n\rfloor$ holds, which proves that for all $n\in \N$, the upper bound above is
  attained for some lattice path of length $n$.
\end{proof}

The following results are immediate consequences of Proposition~\ref{prop:lat-coef-r}.

\begin{corollary}
  \label{cor:lat-random-explicit}
  The probability that a lattice path of length $n$ has reduction degree $r$ is
  given by the explicit formula
  \[ \P(D_{n} = r) = \frac{[z^{n}] L_{r}^{=}(z)}{4^{n}} = 4^{r+1-n} \sum_{\lambda \geq 0}
    \lambda (-1)^{\lambda - 1} \bigg[\binom{2n-1}{n - \lambda 2^{r}} - \binom{2n-1}{n -
      \lambda2^{r} -1}\bigg],  \]
  and the expected reduction degree for paths of length $n$ is given by
  \begin{equation}\label{eq:lat-exp-explicit}
    \E D_{n} = \sum_{k \geq 1} 8k(2^{v_{2}(k)} - 1) \bigg[\binom{2n-1}{n -
    k} - \binom{2n-1}{n - k - 1}\bigg].
  \end{equation}
\end{corollary}
\begin{proof}
  Analogously to our approach in Section~\ref{sec:all-branches}, the double sum
  \[ \E D_{n} = \sum_{r,\lambda \geq 0} 4^{r+1-n} r (-1)^{\lambda - 1} \lambda
    \bigg[\binom{2n-1}{n - \lambda 2^{r}} - \binom{2n-1}{n - \lambda 2^{r} - 1}\bigg]  \]
  can be simplified by considering
  \[ \psi(k) := 4 \sum_{\substack{\lambda, r\geq 0 \\ \lambda 2^{r} = k}} 4^{r} r
    (-1)^{\lambda -1} \lambda. \]
  We find
  \begin{align*}
    \psi(k) & = 4k \bigg(2^{v_{2}(k)} v_{2}(k) - \sum_{r=0}^{v_{2}(k) - 1} r 2^{r}\bigg) =
              8k(2^{v_{2}(k)} - 1),
  \end{align*}
  which proves~\eqref{eq:lat-exp-explicit}.
\end{proof}
\begin{remark}
  The formula for $\P(D_{n} = r)$ is very similar to the results for the classical
  register function obtained by Flajolet (cf.~\cite{these-d'etat}). It is likely that
  applying the techniques that were used
  in~\cite{Louchard-Prodinger:2008:register-lattice} could be used to determine expansions
  for arbitrary moments.
\end{remark}

The following theorem characterizes the asymptotic behavior of the expected
reduction degree and the corresponding variance.

\begin{theorem}
  \label{thm:lat-results}
  The expected reduction degree of simple two-dimensional lattice paths of length
  $n$ admits the asymptotic expansion
  \begin{equation}\label{eq:lat-exp-expansion}
  \E D_{n} = \log_{4} n + \frac{\gamma + 2 - 3\log 2}{2 \log 2} +
    \delta_{1}(\log_{4} n) + O(n^{-1}),
  \end{equation}
  and for the corresponding variance we have
  \begin{multline}\label{eq:lat-var-expansion}
    \V D_{n} = \frac{\pi^{2} - 24\log^{2}\pi - 48 \zeta''(0) - 24}{24 \log^{2} 2} -
    \frac{2 \log\pi}{\log 2} - \frac{11}{12} + \delta_{2}(\log_{4}n) \\ - \frac{\gamma + 2 -
      3\log 2}{\log 2} \delta_{1}(\log_{4} n) - \delta_{1}^{2}(\log_{4} n) +
    O\Big(\frac{\log n}{n}\Big)
  \end{multline}
  where $\delta_{1}(x)$ and $\delta_{2}(x)$ are $1$-periodic fluctuations of mean zero
  which are defined as
    \begin{equation}\label{eq:lat-exp-fluc}
  \delta_{1}(x) = \log 2 \sum_{k\neq 0} c_{k} e^{2k\pi i x}
  \end{equation}
  and
  \begin{equation}\label{eq:lat-var-fluc}
    \delta_{2}(x) = \sum_{k\neq 0} \Big(d_{k} - c_{k}\psi\Big(1 +
      \frac{\chi_{k}}{2}\Big)\Big) e^{2k\pi i x}
  \end{equation}
  with $\chi_{k} =\frac{2\pi i k}{\log 2}$ and constants
  \[ c_{k} = \frac{2}{\sqrt{\pi}\, \log^{2}2 } \Gamma\Big(\frac{3+\chi_{k}}{2}\Big) \zeta(1 + \chi_{k})\] 
  and
  \[ d_{k} = \frac{4}{\sqrt{\pi}\, \log^{2} 2} \Gamma\Big(\frac{3 + \chi_{k}}{2}\Big)
    \big(\psi(2+\chi_{k}) \zeta(1 + \chi_{k}) + \zeta'(1 + \chi_{k})\big) - 3c_{k} \log 2. \]
\end{theorem}
\begin{proof}
In order to analyze the expected value $\E D_{n}$ asymptotically, we study
the corresponding generating function $G^{(1)}(z) = \sum_{r\geq 0} r L_{r}^{=}(z)$,
for which we have $\E D_{n} = \frac{1}{4^{n}} [z^{n}] G^{(1)}(z)$, with an approach that is
similar to the one in Theorem~\ref{thm:asy-branches}.

With the substitution $u = e^{-t}$, we find
\[ G^{(1)}(z) = \sum_{r \geq 1} r 4^{r+1} \frac{u^{2^{r}}}{(1 + u^{2^{r}})^{2}} = \sum_{r,
    \lambda \geq 1} r 4^{r+1} (-1)^{\lambda - 1} \lambda e^{-t \lambda 2^{r}},  \]
where we used~\eqref{eq:expansion-1}. Thus, the Mellin transform $g^{(1)}(s) =
\mathcal{M}(G^{(1)})(s)$ of $G^{(1)}$ (which is a function in $t$) is given by
\begin{align*}
  g^{(1)}(s) & = \sum_{r, \lambda\geq 1} r 4^{r+1} (-1)^{\lambda -1} \lambda^{1-s}
  2^{-rs} \Gamma(s) = 4 \bigg(\sum_{r\geq 1} r 2^{(2-s)r} \bigg) \bigg(\sum_{\lambda \geq
    1} (-1)^{\lambda - 1} \lambda^{1-s}\bigg) \Gamma(s)\\
               & = 4 \frac{2^{2-s}}{(1 - 2^{2-s})^{2}} (1 - 2^{2-s}) \zeta(s-1) \Gamma(s)
                 = 4 \Gamma(s) \zeta(s-1) \frac{2^{2-s}}{1 - 2^{2-s}},
\end{align*}
which is analytic for $\Re(s) > 2$. Observe that $g^{(1)}(s)$ has a pole of order two at
$s = 2$, simple poles at $s = 2 + \chi_{k}$ for $k \in \Z\setminus\{0\}$ and further
simple poles at $s \in -2\mathbb{N}_{0}$.

As the fundamental strip of $g^{(1)}(s)$ is given by $\langle 2, \infty\rangle$, the
Mellin inversion formula yields
\[ G^{(1)}(z) = \frac{1}{2\pi i} \int_{5-i\infty}^{5 + i\infty} g^{(1)}(s) t^{-s}~ds, \]
and we compute this integral by shifting the line of integration to $\Re(s) = -3$.

Note that analogously to the argumentation in the proof of Theorem~\ref{thm:asy-branches},
the Mellin inversion formula above is also valid for complex $z$ in a punctured
neighborhood of $1/4$ where $\abs{\arg(4z - 1)} > 2\pi/5$, which allows us to apply singularity
analysis.

We compute the contributions of the singularities with the help of
SageMath~\cite{SageMath:2016:7.4}. With an analogous estimation as in the proof of
Theorem~\ref{thm:asy-branches} we find that the integral (after the shift) contributes an
error of $O(\abs{t}^{3})$. Again, for the sake of simplicity we take the contribution of the
residue at $-2$ as the error term:
\[ \Res_{s=-2}(g^{(1)}(s)t^{-s}) = O(t^{2}).  \]
Thus, with $P = \{-2, 0, 2\}\cup \{\chi_{k} \mid k\in
\Z\setminus\{0\}\}$ we find
\begin{multline*} 
\sum_{p\in P}\Res_{s=p}(g^{(1)}(s) t^{-s}) = -\frac{4}{\log 2} t^{-2} \log t +
  \Big(\frac{4}{\log 2} - 2\Big)t^{-2} \\ + \frac{4}{9} + \frac{4}{\log 2} \sum_{k\neq
    0} \Gamma(2+\chi_{k}) \zeta(1 + \chi_{k}) t^{-2-\chi_{k}} + O(t^{2}).
\end{multline*}

Substituting back, controlling the error analogously to~\eqref{eq:error-sum-estimate},
applying singularity analysis, normalizing by $4^{n}$, and rewriting the coefficients of
the terms of growth $n^{\chi_{k}/2}$ with the duplication formula for the gamma function
(cf.~\DLMF{5.5}{5}) then proves \eqref{eq:lat-exp-expansion} and \eqref{eq:lat-exp-fluc}.

For the analysis of the variance we turn our attention to the second moments, $\E
D_{n}^{2}$. The related generating function is given by
\[ G^{(2)}(z) = \sum_{r\geq 0} r^{2} L_{r}^{=}(z) = \sum_{r,\lambda \geq 0} r^{2} 4^{r+1}
  (-1)^{\lambda - 1} \lambda e^{-t \lambda 2^{r}}.  \]
It is easy to check that the corresponding Mellin transform $g^{(2)}(s)$ is
\[ g^{(2)}(s) = 4 \Gamma(s) \zeta(s-1) \frac{(1 + 2^{2-s}) 2^{2-s}}{(1 -
    2^{2-s})^{2}}, \]
with a pole of order $3$ at $s = 2$, and poles of order two at $s = 2 + \chi_{k}$ for $k
\in \Z\setminus \{0\}$, as well as simple poles at $s\in -2\N_{0}$. Analogously to above,
the inversion formula is also valid for complex $z$ in a punctured neighborhood around
$1/4$ with $\abs{\arg(1-4z)} > 2\pi/5$. We shift the line of integration to $\Re(s) = -3$,
which yields an error term of
\[ \frac{1}{2\pi i} \int_{-3-i\infty}^{-3+i\infty} g^{(2)}(s)t^{-s}~ds = O(|t|^{3}),  \]
and collect residues.

We find that $\Res_{s=0} (g^{(2)}(s) t^{-s})$ does not yield a
contribution in terms of $z$. The pole at $s = -2$ is the leftmost pole we shift the
line of integration over. For the sake of simplicity, we use the contribution of the
residue at this pole as the error term, which we find to be
\[ \Res_{s=-2}(g^{(2)}(s) t^{-s}) = O(|t|^{2}).  \]

Furthermore, the pole at $s = 2$ yields a residue of
\begin{multline*}
  \Res_{s=2}(g^{(2)}(s) t^{-s}) =  \frac{4}{\log^{2}
    2} t^{-2} \log^{2} t +
  \Big(\frac{4}{\log 2} - \frac{8}{\log^{2} 2}\Big) t^{-2} \log t \\
  + \Big(\frac{\pi^{2} - 12\log^{2} \pi - 24
    \zeta''(0)}{3 \log^{2} 2} - \frac{8 \log\pi + 4}{\log 2} - \frac{8}{3}\Big) t^{-2}
\end{multline*}
which translates into a local expansion of
\begin{multline*}
  \frac{1}{4 \log^{2} 2} \frac{\log^{2}(1 - 4z)}{1 - 4z} + \Big(\frac{3}{2 \log 2} -
  \frac{1}{\log^{2} 2}\Big) \frac{\log(1 - 4z)}{1 - 4z} \\ + \Big(\frac{\pi^{2} - 12 \log^{2}\pi - 24 \zeta''(0)}{12\log^{2} 2} - \frac{2 \log \pi +
  3}{\log 2} + \frac{4}{3}\Big) \frac{1}{1 - 4z} \\
+ O(\log^{2}(1-4z)),
\end{multline*}
and, after applying singularity analysis and dividing by $4^{n}$, into an asymptotic contribution of
\begin{multline*}
  \log_{4}^{2}n + \frac{\gamma + 2 - 3\log 2}{\log 2} \log_{4} n +
  \frac{6\gamma^{2} + \pi^{2} - 24 \log^{2} \pi + 24 \gamma - 48 \zeta''(0)}{24 \log^{2}
    2} \\ - \frac{3 \gamma + 4 \log\pi + 6}{2\log 2} + \frac{4}{3} + O(n^{-1} \log n).
\end{multline*}
Observe that the logarithmic terms in this expansion cancel against the square of the
expansion for $\E D_{n}$ as given in~\eqref{eq:lat-exp-expansion}---which is a common
phenomenon.

Next we determine the contribution of the remaining poles. Locally expanding the sum of
the corresponding residues in terms of $z\to 1/4$ and controlling the resulting error
analogously to~\eqref{eq:error-sum-estimate} yields
\begin{multline*}
\sum_{k\neq 0} \Res_{s = 2+\chi_{k}}(g^{(1)}(s) t^{-s}) = \sum_{k\neq 0} \Big(-
  \log(1 - 4z) (1 - 4z)^{-1 - \chi_{k}/2} \Gamma(1 + \chi_{k}/2)c_{k} \\ + (1 - 4z)^{-1 -
    \chi_{k}/2} \Gamma(1 + \chi_{k}/2) d_{k}\Big) + O(\log(1-4z))
\end{multline*}
where the coefficients $c_{k}$ and $d_{k}$ are defined as in the theorem. Note that
all estimates still work out as the product of the gamma and the digamma function
decays exponentially and the derivative of the zeta function grows at most polynomially,
which is easy to see by considering the derivative by means of Cauchy's integral formula.

Singularity analysis and dividing by $4^{n}$ then yields an asymptotic contribution of
\[ 2\delta_{1}(\log_{4} n) \log_{4} n + \delta_{2}(\log_{4} n) \]
with $\delta_{1}$ and $\delta_{2}$ from \eqref{eq:lat-exp-fluc} and
\eqref{eq:lat-var-fluc}, respectively.

Putting everything together we find
\begin{multline*}
  \E D_{n}^{2} = \log_{4}^{2} n + \frac{\gamma + 2 - 3\log 2}{\log 2}
  \log_{4} n + 2 \delta_{1}(\log_{4} n) \log_{4} n  \\ + \frac{6\gamma^{2} + \pi^{2} -
    24\log^{2} \pi + 24\gamma  -
    48\zeta''(0)}{24\log^{2} 2} \\ - \frac{3\gamma + 4\log\pi + 6}{2\log 2} + \frac{4}{3} +
  \delta_{2}(\log_{4} n) + O\Big(\frac{\log n}{n}\Big).
\end{multline*}

With this result, we are able to find an expansion of the variance $\V D_{n}$ by
considering the difference $\E D_{n}^{2} - (\E D_{n})^{2}$, which
yields the expansion given in~\eqref{eq:lat-var-expansion}.
\end{proof}

\subsection{Fringes}\label{sec:fringes}

We define the $r$th \emph{fringe} of a given lattice path $\ell$ of length
$\geq 1$ to be $\Phi_{L}^{r}(\ell)$, i.e.\ the $r$th fringe is given by the $r$th
reduction of the path. In particular, if $\ell$ can be reduced $r$ times, we call the
length of $\Phi_{L}^{r}(\ell)$ the size of the $r$th fringe. Otherwise, we say that this
size is $0$. We model the size of the $r$th fringe with the random variable
$X_{n;r}^{L}\colon \mathcal{L}_{n} \to \N_{0}$.

The $r$th fringes of positive size can then be enumerated by the bivariate generating function
\begin{equation*}
  \FGF_r(z, v)=\sum_{\substack{\ell\text{ path} \\\rdeg(\ell)\ge r}} v^{\abs{\Phi_L^r(\ell)}}z^{\abs{\ell}}
\end{equation*}
where $\abs{\ell}$ denotes the length of a lattice path.

Deriving a recursion for these generating functions is pretty straightforward: first,
observe that for $r = 0$, the exponent of $v$ always coincides with the exponent of
$z$ as $\Phi_{L}^{0}(\ell) = \ell$ for all lattice paths $\ell$ of length $\geq 1$. Thus
\[ H_{0}(z, v) = L(zv) = \frac{4zv}{1 - 4zv},  \]
where $L(z)$ is the generating function counting all paths of length $\geq 1$.

The recursion itself follows from the fact that $r$th fringes of a path $\ell$ are
$(r-1)$th fringes of its reduction $\Phi_{L}(\ell)$. Thus, by the same argument that was
used in the proof of Proposition~\ref{prop:lat-relation}, we have
\begin{equation}\label{eq:r-fringe-recursion}
  \FGF_r(z, v) = 4 \FGF_{r-1}\Bigl(\Bigl(\frac{z}{1-2z}\Bigr)^2, v\Bigr).
\end{equation}

In this recursion the second parameter, $v$, does not change. This justifies the
application of Proposition~\ref{prop:recursion-solution} in order to rewrite $H_{r}(z, v)$
by means of the substitution $z = Z(u)$. We obtain
\begin{equation*}
  \FGF_r(z, v) =
  \frac{4^{r+1}\frac{u^{2^r}}{(1+u^{2^r})^2}v}{1-\frac{4u^{2^r}}{(1+u^{2^r})^2}v}=
  \frac{4^{r+1}u^{2^r}v}{(1+u^{2^r})^2-4u^{2^r}v}.
\end{equation*}

The generating function $H_{r}(z,v)$ can now be used to derive the asymptotic behavior of
the expectation $\E X_{n;r}^{L}$ and the variance $\V X_{n;r}^{L}$ of the size of the $r$th
fringe, where all paths of length $n$ arise with the same probability.

The first few of those generating functions are
\begin{align*}
\FGF_0(z, v)&=4 vz + 16 v^{2}z^{2} + 64 v^{3}z^{3} + 256 v^{4}z^{4} + 1024
v^{5}z^{5} + 4096 v^{6}z^{6} \\
&\qquad+ 16384 v^{7}z^{7} + 65536 v^{8}z^{8} + 262144 v^{9}z^{9} + O(z^{10}),\\
\FGF_1(z, v)&=16 vz^{2} + 64 vz^{3} + (64 v^{2} + 192 v)z^{4} + (512 v^{2} +
512 v)z^{5}\\
&\qquad+ (256 v^{3} + 2560 v^{2} + 1280 v)z^{6} + (3072 v^{3} + 10240 v^{2} +
3072 v)z^{7}\\
&\qquad + (1024 v^{4} + 21504 v^{3} + 35840 v^{2} + 7168 v)z^{8}\\
&\qquad + (16384 v^{4} + 114688 v^{3} + 114688 v^{2} + 16384 v)z^{9} + O(z^{10}),\\
\FGF_2(z, v)&=64 vz^{4} + 512 vz^{5} + 2816 vz^{6} + 13312 vz^{7} + (256 v^{2}
+ 58112 v)z^{8}\\
&\qquad + (4096 v^{2} + 241664 v)z^{9} + O(z^{10}),\\
\FGF_3(z, v)&=256 vz^{8} + 4096 vz^{9} + O(z^{10}).
\end{align*}

In order to get a better understanding of the behavior of fringe sizes, we investigate the
minimum and maximum value of the random variable $X_{n;r}^{L}$ modeling the size of the $r$th
fringe of a random lattice path of length $n$.

\begin{proposition}\label{prop:r-fringes-range}
  Let $n$, $r\in \N_{0}$. If $r = 0$, then $X_{n;0}^{L}$ is a deterministic quantity with
  $X_{n;0}^{L} = n$. For $r > 0$, the bound
  \[ \llbracket n > 0 \text{ and } r=1 \rrbracket \leq X_{n;r}^{L} \leq \Big\lfloor
    \frac{n}{2^{r}}\Big\rfloor  \]
  holds and is sharp.
\end{proposition}
\begin{proof}
  The proof follows the same idea as the proof of Proposition~\ref{prop:r-branches-range}. In
  particular, the concept of ``minimal expansion'' of binary trees corresponds to
  expanding single steps into segments of length $2$. Furthermore, an appropriate family
  of lattice paths can be constructed from the steps $\{\upa, \righta, \downa,
  \lefta\}$ by iteratively expanding the path either minimally, or expanding the first
  step into a segment of length $3$ and the rest minimally.
\end{proof}

\begin{theorem}\label{thm:fringe-sizes}
  Let $r\in \N_{0}$ be fixed. The expectation and variance of the size of the $r$th fringe
  of a random path of length $n$ have the asymptotic expansions
  \begin{equation}
    \label{eq:fringes-exp}
    \E X_{n;r}^{L} = \frac{n}{4^{r}} + \frac{1 - 4^{-r}}{3} + O(n^{3} \theta_{r}^{-n})
  \end{equation}
  and
  \begin{equation}
    \label{eq:fringes-var}
    \V X_{n;r}^{L} = \frac{4^{r} - 1}{3\cdot 16^{r}} n + \frac{-2\cdot 16^{r} - 5\cdot 4^{r}
      + 7}{45\cdot 16^{r}} + O(n^{5} \theta_{r}^{-n}),
  \end{equation}
  where $\theta_{r} = \frac{4}{2 + 2\cos(2\pi/2^{r})} > 1$.
  If additionally $r > 0$, then for the random variables $X_{n;r}^{L}$
  modeling the $r$th fringe size of lattice paths of length $n$ we have
  \[ \P\bigg(\frac{X_{n;r}^{L} - \E X_{n;r}^{L}}{\sqrt{\V X_{n;r}^{L}}} \leq x\bigg) = \frac{1}{\sqrt{2\pi}}
    \int_{-\infty}^{x} e^{-w^{2}/2}~dw + O(n^{-1/2}),  \]
  i.e.\ the random variables $X_{n;r}^{L}$ are asymptotically normally distributed.
\end{theorem}
\begin{proof}
  The generating function $H_{r}(z, v)$ only sums over all lattice paths with
  reduction degree $\geq r$. In a first step we show that the number of excluded
  paths is exponentially small when compared with the number of all paths.

  In order to do so, we consider the generating function
  \[ H_{r}(z, 1) = \frac{4^{r+1} u^{2^{r}}}{(1 - u^{2^{r}})^{2}}.  \]
  From~\eqref{eq:r-fringe-recursion} we know that $H_{r}(z,1)$ is a meromorphic function,
  i.e.\ all its singularities are poles and no square-root singularities can occur.
  In the $u$-world the singularities can be expressed as the $2^{r}$th roots of
  unity. From the proof of Proposition~\ref{prop:subs-prop} we also know that on the unit
  circle,
  \[ Z(u) = \frac{1}{2 + 2\Re u}  \]
  holds, such that with Property~(\ref{prop:subs-prop:pole-translation}) of
  Proposition~\ref{prop:subs-prop} we know that the dominant singularity (in terms of $z =
  Z(u)$) is a simple pole at $z = Z(1) = 1/4$, and the next singularity is a pole
  of order two at 
  \[ \frac{\theta_{r}}{4} := Z(e^{- 2\pi i/2^{r}}) = \frac{1}{2 + 2\cos(2\pi /2^{r})},  \]
  which translates into a contribution of $O(n 4^{n}\theta_{r}^{-n})$.

  Together with the local expansion
  \[ H_{r}(z, 1) = \frac{1}{1 - 4z} + O(1) \]
  for $z\to 1/4$, and with the fact that $H_{r}(z, 1)$ is meromorphic, we find
  that
  \begin{equation}\label{eq:wrong-model}
    [z^{n}] H_{r}(z,1) = 4^{n} + O(n (2 + 2\cos(2\pi/2^{r}))^{n}),
  \end{equation}
  as claimed. For determining the moments, fringes of size $0$ do not yield a contribution
  (as they are weighted with $0$), such that we can use the generating function $H_{r}(z,
  v)$.

  It is easy to see that $\E X_{n;r}^{L}$ can be obtained by dividing the coefficient of
  $z^{n}$ in $\big(\frac{\partial H_{r}}{\partial v} (z, v)\big)\big|_{v=1}$ by the
  normalization factor. In particular, we find
  \[ \Big(\frac{\partial H_{r}}{\partial v} (z, v)\Big)\Big|_{v=1} = 4^{r+1} u^{2^{r}} \frac{(1 +
      u^{2^{r}})^{2}}{(1 - u^{2^{r}})^{4}} = \frac{4^{-r}}{(1-4z)^{2}} + \frac{1 -
      4^{1-r}}{3(1-4z)} + O(1). \]
  Applying singularity analysis to this meromorphic function and dividing by $4^{n}$
  yields~\eqref{eq:fringes-exp}.

  For the variance we compute asymptotic expansions for the second moment by
  considering the generating function
  \[ \Big(\frac{\partial^{2} H_{r}}{\partial v^{2}}(z, v) + \frac{\partial H_{r}}{\partial
      v}(z,v)\Big)\Big|_{v=1} = \frac{2\cdot 16^{-r}}{(1-4z)^{3}} + \frac{4^{-r} - 4\cdot 16^{-r}}{(1 -
      4z)^{2}} + \frac{1 -20\cdot 4^{-r} + 34\cdot 16^{-r}}{15 (1-4z)} + O(1). \]
  In this case, singularity analysis and normalization leads to a contribution of
  \[ \frac{n^{2}}{16^{r}} + \frac{4^{r} - 1}{16^{r}} n + \frac{16^{r} - 5\cdot 4^{r} +
      4}{15 \cdot 16^{r}} + O(n^{5} \theta_{r})^{-n} \]
  for the second moment. Subtracting $(\E X_{n;r}^{L})^{2}$ from this
  expansion results in~\eqref{eq:fringes-var}.

  For the limiting distribution we restrict our model to lattice paths admitting $r$
  reductions and study the corresponding random variable $\tilde
  X_{n;r}^{L}$. By~\eqref{eq:wrong-model} this induces an exponentially small error in the
  sense that
  \[ \P(\tilde X_{n;r}^{L} \in A) = \P(X_{n;r}^{L} \in A) \big(1 + O(\theta_{r}^{-n})\big)  \]
  for all $A \subseteq \N_{0}$.

  By singularity perturbation of meromorphic functions (cf.~\cite[Theorem IX.9]{Flajolet-Sedgewick:ta:analy})
  we immediately find that $\tilde X_{n;r}^{L}$ is asymptotically normally distributed---and
  as a direct consequence of the exponentially small error observed above, $X_{n;r}^{L}$ is
  asymptotically normally distributed as well.
\end{proof}

As we have the generating function $H_{r}(z,v)$ in an explicit form, the expected value
can also be extracted explicitly by means of Cauchy's integral formula.

\begin{proposition}\label{prop:fringes-explicit}
  For given $r\in \N_{0}$, the expected size of the $r$th fringe $\E X_{n;r}^{L}$ of a
  random path of length $n$ is given by the explicit formula
  \[ \E X_{n;r}^{L} = 4^{r+1-n} \sum_{\lambda \geq 1} \frac{2\lambda^{3} + \lambda}{3}
    \bigg[\binom{2n-1}{n-2^{r}\lambda} - \binom{2n-1}{n - 2^{r}\lambda -1}\bigg].  \]
\end{proposition}
\begin{proof}
Applying Cauchy's integral formula to
\[ \E X_{n;r}^{L} = 4^{-n} [z^{n}] \Big(\frac{\partial H_{r}}{\partial v}(z, v)\Big)\Big|_{v=1} \]
yields
\begin{align*}
\E X_{n;r}^{L}&= 4^{-n} [z^n]4^{r+1}u^{2^r}\frac{(1+u^{2^r})^2}{(1-u^{2^r})^4} = \frac{4^{r+1-n}}{2\pi i}\oint
         u^{2^r}\frac{(1+u^{2^r})^2}{(1-u^{2^r})^4} \frac{dz}{z^{n+1}}\\
&=\frac{4^{r+1-n}}{2\pi i}\oint u^{2^r}\frac{(1+u^{2^r})^2}{(1-u^{2^r})^4}
  \frac{(1-u)(1+u)^{2n+2}}{(1+u)^3} \frac{du}{u^{n+1}}\\
&= 4^{r+1-n} [u^n](1-u)(1+u)^{2n-1}u^{2^r}\frac{(1+u^{2^r})^2}{(1-u^{2^r})^4}\\
&=4^{r+1-n}\sum_{\lambda\ge1}\frac{2\lambda^3+\lambda}{3}[u^n](1-u)(1+u)^{2n-1}u^{2^r\lambda},
\end{align*}
where the last step is justified by
\begin{equation}\label{eq:geom-x4}
x \frac{(1+x)^{2}}{(1-x)^{4}} = \sum_{\lambda\geq 1} \frac{2 \lambda^{3} +
    \lambda}{3} x^{\lambda}.
\end{equation}
The statement of the theorem then follows by extracting the coefficients
by means of the binomial theorem.
\end{proof}

Analogously to our investigations concerning branches in binary trees, we also
study the behavior of the overall fringe size, i.e.\ the sum over the size of the $r$th
fringes for $r\geq 0$. Like the reduction degree, this parameter can also be interpreted
as a complexity measure for lattice paths. We will model this quantity with the random
variable $X_{n}^{L} := \sum_{r\geq 0} X_{n;r}^{L}$.

A first observation regarding the behavior of $X_{n}^{L}$ can be followed directly from
Proposition~\ref{prop:r-fringes-range}.

\begin{proposition}\label{prop:paths:total-fringe-range}
  Let $n\in \N_{0}$ and let $w_{2}(n)$ denote the binary weight, i.e.\ the number of
  non-zero digits in the binary expansion of $n$. Then the bound
  \[ n + \llbracket n > 1\rrbracket \leq X_{n}^{L} \leq 2n - w_{2}(n) \leq 2n - 1   \]
  holds and is sharp.
\end{proposition}
\begin{proof}
  Analogously to the proof of Proposition~\ref{prop:total-branches-bounds}.
\end{proof}

Furthermore, summing the explicit expressions for $\E X_{n;r}^{L}$ obtained in
Proposition~\ref{prop:fringes-explicit} yields an explicit formula for $\E X_{n}^{L}$, the
expected fringe size for lattice paths of length $n$.

\begin{corollary}\label{cor:fringe-all-explicit}
  The expected fringe size $\E X_{n}^{L}$ of a random path of length $n$
  can be computed as
  \[ \E X_{n}^{L} = \frac{1}{12\cdot 4^{n}} \sum_{k=1}^{n} \big(2k^{3} (2 - 2^{-v_{2}(k)}) +
    k(2^{v_{2}(k) + 1} - 1)\big)\bigg[\binom{2n-1}{n-k} - \binom{2n-1}{n-k-1}\bigg].  \]
\end{corollary}
\begin{proof}
  Analogously to the proof of Corollary~\ref{cor:explicit-exp-branches}.
\end{proof}

The following theorem quantifies the asymptotic behavior of $\E X_{n}^{L}$.

\begin{theorem}\label{thm:fringe-size}
  Asymptotically, the behavior of the expected fringe size $\E X_{n}^{L}$ for a random path
  of length $n$ is given by
  \begin{equation}
    \label{eq:fringe-size-exp}
    \E X_{n}^{L} = \frac{4}{3} n + \frac{1}{3} \log_{4}n + \frac{5 + 3\gamma - 11\log 2}{18
      \log 2} + \delta(\log_{4} n) + O\Big(\frac{\log n}{n}\Big),
  \end{equation}
  where $\delta(x)$ is a $1$-periodic fluctuation of mean zero with Fourier series
  expansion
  \[ \delta(x) = \frac{2}{3\sqrt{\pi} \log 2} \sum_{k\neq 0} \Gamma\Big(\frac{3 + \chi_{k}}{2}\Big)
    \bigl(2\zeta(\chi_{k} - 1) + \zeta(\chi_{k} + 1)\bigr) e^{2k\pi i x}.  \]
\end{theorem}
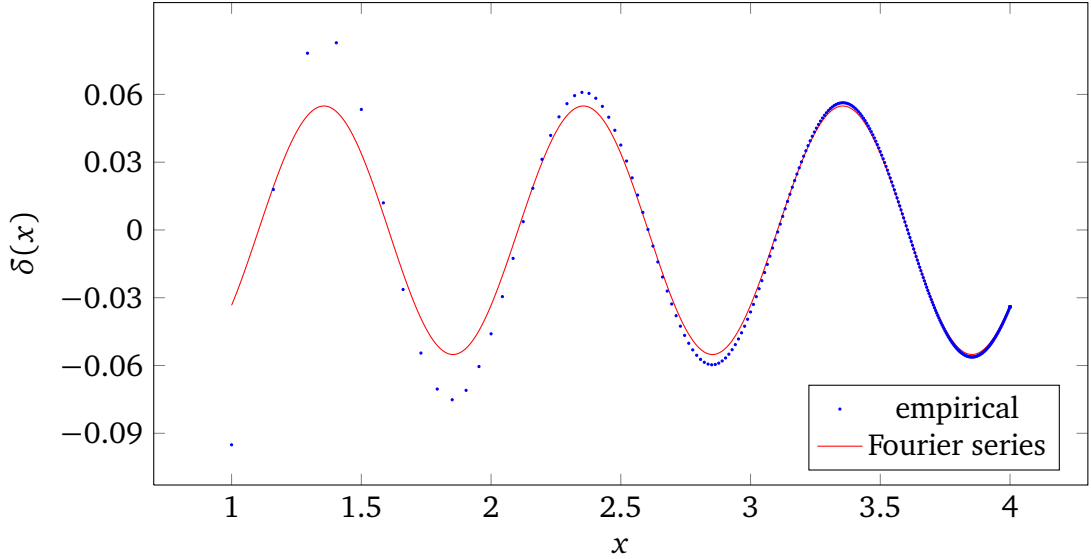
\begin{figure}[ht]
  \centering
  \begin{tikzpicture}
    \begin{axis}[xlabel={$x$}, ylabel={$\delta(x)$}, width=14cm, height=8cm,
      yticklabel style={/pgf/number format/fixed, /pgf/number format/precision=3},
      xtick = {1,1.5,...,4}, legend entries = {empirical, Fourier series},
      ytick = {-0.09, -0.06, ..., 0.06}, legend pos = south east]
      \addplot+[only marks, mark size=0.4pt] table[x=x2, y=empirical] {theorem5_fluc.dat};
      \addplot+[no marks] table[x=x, y=fourier] {theorem5_fluc.dat};
    \end{axis}
  \end{tikzpicture}
  \caption{Partial Fourier series (20 summands) compared with the empirical values of the
    function $\delta$ from Theorem~\ref{thm:fringe-size}}
  \label{fig:fringe-size-fluc}
\end{figure}
\begin{proof}
  We follow the strategy from the proof of Theorem~\ref{thm:asy-branches}. First of all,
  observe that with the substitution $u = e^{-t}$,
  \[ H^{(1)}(z) := \sum_{r\geq 0} \Big(\frac{\partial H_{r}}{\partial v}(z, v)\Big)\Big|_{v=1} = \sum_{r\geq 0} 4^{r+1} u^{2^{r}} \frac{(1 +
      u^{2^{r}})^{2}}{(1 - u^{2^{r}})^{4}} = \frac{4}{3} \sum_{\substack{r\geq 0,\\ \lambda\geq 1}} 4^{r}
    (2\lambda^{3} + \lambda) e^{-t \lambda 2^{r}}, \]
  where we used~\eqref{eq:geom-x4}.
  The Mellin transform $h^{(1)}(s) := \mathcal{M}(H^{(1)})(s)$ is then easily
  determined, we find
  \begin{align*}
  h^{(1)}(s) &= \frac{4}{3} \sum_{\substack{r\geq 0,\\ \lambda\geq 1}} 4^{r}
    (2\lambda^{3} + \lambda) \lambda^{-s} 2^{-rs} \Gamma(s) = \frac{4}{3} \Gamma(s)
    \sum_{r\geq 0} (2^{2-s})^{r} \sum_{\lambda\geq 1} (2 \lambda^{3-s} +
    \lambda^{1-s}) \\
    & = \frac{4}{3} \Gamma(s) \frac{2 \zeta(s-3) + \zeta(s-1)}{1 - 2^{2-s}},
  \end{align*}
  a function that is analytic for $\Re(s) > 4$. Thus, by the Mellin inversion formula, we
  have
  \[ H^{(1)}(z) = \frac{1}{2\pi i} \int_{5-i\infty}^{5+i\infty} h^{(1)}(s) t^{-s}\,ds,  \]
  again valid in a punctured complex neighborhood of $1/4$ with $\abs{\arg(4z-1)} > 2\pi/5$.
  
  With an analogous justification as in the proof of Theorem~\ref{thm:asy-branches} and
  the proof of Theorem~\ref{thm:lat-results} we
  shift the line of integration to the line $\Re(s) = -3$ and take the contribution from the
  residue at $-2$ as the error term (which gives an expansion error of
  $O(1-4z)$). Analogously to the previous theorems, the remaining integral is absorbed by
  this error term.

  By shifting the line of integration we cross a simple pole at $s = 4$, a pole of order
  two at $s = 2$, infinitely many simple poles at $s = 2 + \chi_{k}$, $k \in
  \Z\setminus\{0\}$, and a simple pole in $s = 0$.

  With $P = \{-2, 0, 2, 4\} \cup \{2+\chi_{k} \mid k\in \Z\setminus \{0\}\}$ we find that
  \begin{multline*} 
    \sum_{p\in P} \Res_{s = p}(h^{(1)}(s) t^{-s}) = \frac{64}{3}t^{-4} + \Big(\frac{2}{3} +
    \frac{10}{9 \log 2}\Big) t^{-2} - \frac{4 \log t}{3\log 2} t^{-2} + \frac{4}{135} \\*
    + \frac{4}{3 \log 2} \sum_{k\neq 0} \Gamma(2 + \chi_{k})(2\zeta(-1 + \chi_{k}) + \zeta(1 + \chi_{k}))
    t^{-2-\chi_{k}} + O(t^{2}).
  \end{multline*}
  Local expansion in terms of $z\to 1/4$ and controlling the error of the translation of
  the residues along the vertical line $\Re s = 2$ analogously
  to~\eqref{eq:error-sum-estimate} then results in
  \begin{multline*}
    \frac{4}{3 (1 - 4z)^{2}} - \frac{\log(1 - 4z)}{6 (1 - 4z) \log 2} + \Big(\frac{5}{18 \log
      2} - \frac{35}{18}\Big) \frac{1}{1 - 4z} \\ + \frac{1}{3\log 2} \sum_{k \neq 0} \Gamma(2
    + \chi_{k}) (2\zeta(-1 + \chi_{k}) + \zeta(1 + \chi_{k}))(1 - 4z)^{- 1 - \chi_{k}/2} +
    O(\log(1-4z)),
  \end{multline*}
  from which the statement of the theorem follows after extracting the coefficient growth
  of this expansion by means of singularity analysis, dividing by $4^{n}$, and applying
  the duplication formula for the gamma function.
\end{proof}

\paragraph{\textbf{Acknowledgment}} We thank Michael Fuchs for a hint regarding
the central limit theorem for additive tree parameters.

\bibliographystyle{amsplainurl}
\bibliography{register-reduction,bib/cheub}
\end{document}